\documentclass[11pt,reqno]{amsart}
\pdfoutput=1

\usepackage{latexsym,amsthm,amssymb,amsmath,amscd,hyperref}
\usepackage{enumitem}
\usepackage[all,cmtip]{xy}

\newtheorem{theorem}{Theorem}[section]
\newtheorem{proposition}[theorem]{Proposition}
\newtheorem{lemma}[theorem]{Lemma}
\newtheorem{corollary}[theorem]{Corollary}

\theoremstyle{remark}
\newtheorem{remark}[theorem]{Remark}

\newcommand{\Cl}{\textup{Cl}}
\newcommand{\Prim}{\textup{Prim}}

\newcommand{\Spec}{\textup{Spec}}
\newcommand{\spn}{\textup{span}}

\renewcommand{\a}{\mathfrak{a}}
\renewcommand{\b}{\mathfrak{b}}
\renewcommand{\c}{\mathfrak{c}}
\renewcommand{\k}{\mathfrak{k}}
\newcommand{\m}{\mathfrak{m}}
\newcommand{\n}{\mathfrak{n}}
\newcommand{\p}{\mathfrak{p}}

\newcommand{\I}{\mathcal{I}}
\newcommand{\J}{\mathcal{J}}
\renewcommand{\P}{\mathcal{P}}

\begin{document}

\title[C*-algebras from actions of congruence monoids on rings of integers]{C*-algebras from actions of congruence monoids\\ on rings of algebraic integers} 
\author{Chris Bruce}
\address[Chris Bruce]{
Department of Mathematics and Statistics\\
University of Victoria\\
Victoria, BC V8W 3R4\\
Canada}
\email{cmbruce@uvic.ca}
%\date{\today}

\thanks{Research supported by the Natural Sciences and Engineering Research Council of Canada through an Alexander Graham Bell CGS-D award.\\ This work was done as part of the author's PhD project at the University of Victoria.}

\maketitle

\begin{abstract}
Let $K$ be a number field with ring of integers $R$. Given a modulus $\mathfrak{m}$ for $K$ and a group $\Gamma$ of residues modulo $\mathfrak{m}$, we consider the semi-direct product $R\rtimes R_{\mathfrak{m},\Gamma}$ obtained by restricting the multiplicative part of the full $ax+b$-semigroup over $R$ to those algebraic integers whose residue modulo $\mathfrak{m}$ lies in $\Gamma$, and we study the left regular C*-algebra of this semigroup. We give two presentations of this C*-algebra and realize it as a full corner in a crossed product C*-algebra. We also establish a faithfulness criterion for representations in terms of projections associated with ideal classes in a quotient of the ray class group modulo $\mathfrak{m}$, and we explicitly describe the primitive ideals using relations only involving the range projections of the generating isometries; this leads to an explicit description of the boundary quotient. 
Our results generalize and strengthen those of Cuntz, Deninger, and Laca and of Echterhoff and Laca for the C*-algebra of the full $ax+b$-semigroup.
We conclude by showing that our construction is functorial in the appropriate sense; in particular, we prove that the left regular C*-algebra of $R\rtimes R_{\mathfrak{m},\Gamma}$ embeds canonically into the left regular C*-algebra of the full $ax+b$-semigroup. Our methods rely heavily on Li's theory of semigroup C*-algebras.
\end{abstract}

\section{Introduction}
 
\subsection{Historical context} 
Cuntz pioneered the study of C*-algebras associated with $ax+b$-semigroups over the ring $\mathbb{Z}$ in \cite{Cun}; his work was motivated by the construction of Bost and Connes in \cite{BC}. 
Cuntz introduced a C*-algebra $\mathcal{Q}_\mathbb{N}$ defined using generators and relations involving the additive group of $\mathbb{Z}$ and the multiplicative semigroup $\mathbb{N}^\times:=\mathbb{N}\setminus\{0\}$. The C*-algebra $\mathcal{Q}_\mathbb{N}$ can be canonically (and faithfully) represented on $\ell^2(\mathbb{Z})$, $\mathcal{Q}_\mathbb{N}$ is simple and purely infinite, and admits a unique KMS state for a canonical time evolution, see \cite{Cun}. Cuntz showed that $\mathcal{Q}_\mathbb{N}$ can be realized as a full corner in the crossed product C*-algebra for the action of the $ax+b$-group $\mathbb{Q}\rtimes\mathbb{Q}_+^*$ on the ring $\mathbb{A}_{\mathbb{Q},f}$ of finite adeles over $\mathbb{Q}$ and then discussed its $K$-theory.
Another C*-algebra $\mathcal{Q}_\mathbb{Z}$ was defined in \cite{Cun} using an analogous presentation but with the larger multiplicative semigroup $\mathbb{Z}^\times:=\mathbb{Z}\setminus\{0\}$ of all non-zero integers in place of $\mathbb{N}^\times$.

Laca and Raeburn initiated the study of Toeplitz algebras in this context, see \cite{LR3}. They showed that the semigroup $\mathbb{N}\rtimes\mathbb{N}^\times$ is quasi-lattice ordered, and they studied phase transitions for a canonical time evolution on its left regular C*-algebra $C_\lambda^*(\mathbb{N}\rtimes\mathbb{N}^\times)$ (which they called the ``Toeplitz algebra'' of $\mathbb{N}\rtimes\mathbb{N}^\times$). They also exhibited Cuntz's $\mathcal{Q}_\mathbb{N}$ as the boundary quotient of $C_\lambda^*(\mathbb{N}\rtimes\mathbb{N}^\times)$. In a subsequent paper, Laca and Neshveyev parameterized the Nica spectrum of $\mathbb{N}\rtimes\mathbb{N}^\times$ in terms of an adelic space and computed the type of each equilibrium state at high temperature, see \cite{LN2}.

Building on \cite{Cun}, Cuntz and Li introduced the so-called ring C*-algebras in \cite{CL1} (see also \cite{Li1}). In particular, given a ring of integers $R$ in a number field $K$, Cuntz and Li defined a C*-algebra $\mathfrak{A}[R]$ using generators and relations generalizing those used in \cite{Cun} to define $\mathcal{Q}_\mathbb{Z}$, so that for the ring $\mathbb{Z}$, their construction gave the C*-algebra $\mathcal{Q}_\mathbb{Z}$. They showed that $\mathfrak{A}[R]$ also has a canonical (and faithful) representation on $\ell^2(R)$, and proved that $\mathfrak{A}[R]$ is simple and purely infinite. They gave a description of $\mathfrak{A}[R]$ as a canonical full corner in the crossed product for the action of the $ax+b$-group $K\rtimes K^\times$ on the ring $\mathbb{A}_{K,f}$ of finite adeles over $K$, and used this description to make a connection with Bost-Connes type systems for arbitrary number fields as defined in \cite{LLN}.
The problem of computing the $K$-theory of $\mathfrak{A}[R]$ was particularly difficult; it was solved in the case that $K$ has only two roots of unity by Cuntz and Li in \cite{CL2} using a duality theorem for global fields, and then in full generality by Li and L\"{u}ck in \cite{LiLu}.

Cuntz, Deninger, and Laca defined Toeplitz algebras associated with rings of integers of arbitrary number fields in \cite{CDL}. Given a number field $K$ with ring of integers $R$, they defined a C*-algebra $\mathfrak{T}[R]$ using generators and relations similar to those used to define $\mathfrak{A}[R]$, but without certain ``tightness'' relations. They proved that $\mathfrak{T}[R]$ is canonically isomorphic to the left regular C*-algebra $C_\lambda^*(R\rtimes R^\times)$ of the $ax+b$-semigroup $R\rtimes R^\times$ where the multiplicative semigroup $R^\times:=R\setminus\{0\}$ acts on (the additive group of) $R$ by multiplication.  
In \cite{CDL}, the left regular C*-algebra of $R\rtimes R^\times$ is denoted by $\mathfrak{T}$ and is called the ``Toeplitz algebra'' of $R\rtimes R^\times$. 
Cuntz, Deninger, and Laca studied phase transitions for a canonical time evolution on $C_\lambda^*(R\rtimes R^\times)$, and they proved that the associated C*-dynamical system exhibits several interesting properties. They gave a description of $C_\lambda^*(R\rtimes R^\times)$ as a full corner in a crossed product for an action of the $ax+b$-group $K\rtimes K^\times$ on a certain adelic space, and proved that their construction was functorial for inclusions of rings of integers. They also showed that the ring C*-algebra $\mathfrak{A}[R]$ of $R$ appeared naturally as a quotient of $C_\lambda^*(R\rtimes R^\times)$.

Since \cite{CDL} appeared, the C*-algebras of $ax+b$-semigroups over rings of algebraic integers have been studied intensively. They inspired Neshveyev to prove a powerful general result on KMS states for groupoid C*-algebras, see \cite{Ne}, where Neshveyev also gives an alternative approach to proving the phase transition theorem from \cite{CDL}.
These C*-algebras also provided a motivating class of examples for Li's theory of semigroup C*-algebras developed in \cite{Li1,Li2} (see also \cite[Chapter~5]{CELY}). 
In \cite{EL}, Echterhoff and Laca developed general results on primitive ideal spaces of crossed products, then used these results to compute the primitive ideal space of $C_\lambda^*(R\rtimes R^\times)$. Cuntz, Echterhoff, and Li proved a general formula for the $K$-theory of a large class of semigroup C*-algebras in \cite{CEL1,CEL2} which, as a particular case, gives a formula for the $K$-theory of $C_\lambda^*(R\rtimes R^\times)$. They also showed in \cite{CEL1} that $C_\lambda^*(R\rtimes R^\times)$ is purely infinite, has the ideal property, but does not have real rank zero. 
Building on these works, Li gave an explicit description of the primitive ideals in $C_\lambda^*(R\rtimes R^\times)$ in \cite{Li4} and used $K$-theoretic invariants to show that one can recover the Dedekind zeta function of $K$ from $C_\lambda^*(R\rtimes R^\times)$, provided that one knows the number of roots of unity in $K$. Continuing his investigation, Li showed in \cite{Li5} that one can recover both the Dedekind zeta function of $K$ and the ideal class group $\Cl(K)$ of $K$ from $C_\lambda^*(R\rtimes R^\times)$ together with its canonical diagonal sub-C*-algebra. Li also studied the semigroup C*-algebras of $ax+b$-semigroups for more general classes of rings in \cite{Li6}, where he showed that some of the results on ideal structure, pure infiniteness, and $K$-theory can be generalized; in \cite{Li7}, he gives an alternative approach to pure infiniteness of these $ax+b$-semigroup C*-algebras using partial transformation groupoids.
Recently, Laca and Warren in \cite{LW} have used Neshveyev's characterization of traces on crossed products from \cite[Section~2]{Ne} to describe the low temperature KMS equilibrium states from the phase transition theorem in \cite{CDL} in terms of ergodic invariant measures for groups of linear toral automorphisms. As a result, this revealed a connection with the generalized Furstenberg conjecture in ergodic theory.

\subsection{Overview of the construction}
In this paper, we generalize the construction from \cite{CDL} by considering the C*-algebras of a larger class of semigroups. The construction of these semigroups depends not only on a number field $K$, but also on additional number-theoretic data that arise naturally in the study of the ray class fields of $K$, that is, in class field theory.  
Namely, given a number field $K$ with ring of integers $R$, a modulus $\m$ for $K$, and a group $\Gamma$ of residues modulo $\m$, the associated congruence monoid $R_{\m,\Gamma}$ is the multiplicative monoid of algebraic integers in $R$ that reduce to an element of $\Gamma$ modulo $\m$. We form the semi-direct product $R\rtimes R_{\m,\Gamma}$ where $R_{\m,\Gamma}$ acts on $R$ by multiplication, and investigate the left regular C*-algebra of this semigroup.
We formulate and prove the appropriate generalizations of several of the results mentioned above for the full $ax+b$-semigroup. In addition, we give a new faithfulness criterion for representations, see Section~\ref{sec:faithfulreps}.

We now briefly explain our construction in the special case of the number field $K=\mathbb{Q}$, see Section~\ref{sec:construction} for a detailed discussion of the general case. Let $\P_\mathbb{Q}$ denote the set of rational prime numbers, and let $w$ be the unique embedding $w:\mathbb{Q}\hookrightarrow\mathbb{R}$. A \emph{modulus} for $\mathbb{Q}$ is a function $\m:\{w\}\sqcup\P_\mathbb{Q}\to \mathbb{N}$ such that $\m(w)\in\{0,1\}$ and $\m(p)=0$ for all but finitely many primes $p\in\P_\mathbb{Q}$. Denote by $m$ the positive integer $\prod_{p\in\P_\mathbb{Q}}p^{\m(p)}$. The \emph{multiplicative group of residues modulo $\m$} is $(\mathbb{Z}/\m)^*:=\{\pm 1\}\times(\mathbb{Z}/m\mathbb{Z})^*$ where $(\mathbb{Z}/m\mathbb{Z})^*$ is the multiplicative group of invertible elements in the ring $\mathbb{Z}/m\mathbb{Z}$. For $a\in\mathbb{Z}$ such that $\gcd(a,m)=1$, the \emph{residue of $a$ modulo $\m$} is 
\[
[a]_\m:=(\textup{sign}(a),a+m\mathbb{Z})\in (\mathbb{Z}/\m)^*
\]
where $\textup{sign}(a):=a/|a|$. Dealing with moduli allows us to speak of congruence relations that can involve positivity conditions.
Let $\Gamma\subseteq(\mathbb{Z}/\m)^*$ be a subgroup, and let 
\[
\mathbb{Z}_{\m,\Gamma}:=\{a\in \mathbb{Z}^\times : \gcd(a,m)=1, [a]_\m\in\Gamma\}   
\]
where $\mathbb{Z}^\times:=\mathbb{Z}\setminus \{0\}$. 
Since $\Gamma$ is a group, $\mathbb{Z}_{\m,\Gamma}$ is a unital semigroup under multiplication. Such semigroups are called \emph{congruence monoids}, see \cite[Definition~5]{HK} and \cite{G-HK}. 
Notice that $\mathbb{Z}_{\m,\Gamma}$ is a disjoint union of arithmetic progressions; for example, if $\Gamma$ is the trivial group, then $\mathbb{Z}_{\m,\Gamma}=1+m\mathbb{N}$.
We form the semi-direct product semigroup $\mathbb{Z}\rtimes\mathbb{Z}_{\m,\Gamma}$ with respect to the action of $\mathbb{Z}_{\m,\Gamma}$ on (the additive group of) $\mathbb{Z}$ given by multiplication. The left regular C*-algebra of $\mathbb{Z}\rtimes\mathbb{Z}_{\m,\Gamma}$ is the sub-C*-algebra of $\mathcal{B}(\ell^2(\mathbb{Z}\rtimes\mathbb{Z}_{\m,\Gamma}))$ generated by the isometries $\lambda_{(b,a)}$ for $(b,a)\in\mathbb{Z}\rtimes\mathbb{Z}_{\m,\Gamma}$ defined via the left translation action of $\mathbb{Z}\rtimes\mathbb{Z}_{\m,\Gamma}$ on itself. 
In this article, we study C*-algebras of semigroups of this kind and their analogues for general number fields.

It is very natural to consider C*-algebras associated with semigroups of the form $R\rtimes M$ where $M$ is a subsemigroup of $R^\times$. For $K=\mathbb{Q}$ and $R=\mathbb{Z}$, such C*-algebras have already been considered in two special cases: Larsen and Li in \cite{LarLi}  considered the \emph{2-adic ring C*-algebra} associated with the semigroup $\mathbb{Z}\rtimes [2\rangle$ where $[2\rangle:=\{1,2,2^2,2^3,...\}$, and Barlak, Omland, and Stammeier in \cite{BOS} considered C*-algebras associated with semigroups of the form $\mathbb{Z}\rtimes M$ where $M$ is a subsemigroup of $\mathbb{N}^\times$ generated by a non-empty family of relative prime numbers.
If we consider the special case where $\Gamma=\{1\}\times(\mathbb{Z}/m\mathbb{Z})^*$, then $\mathbb{Z}_{\m,\Gamma}$ is the subsemigroup of $\mathbb{N}^\times$ generated by the prime numbers that do not divide $m$, so that our $\mathbb{Z}\rtimes\mathbb{Z}_{\m,\Gamma}$ is a semigroup of the  type considered in \cite{BOS}.

Some of the analysis in Sections~\ref{sec:construction},~\ref{subsec:presentation},~and~\ref{subsec:gpoidmodel} can likely be generalized to other semigroups of the form $R\rtimes M$. However, results in later sections of this paper rely heavily on $M$ being a congruence monoid, which shows that actions of congruence monoids give rise to particularly nice semigroups, and we thus focus on this case from the beginning to avoid unnecessary technical difficulties. The author plans to consider more general semigroups of the form $R\rtimes M$ in a future work.

\subsection{Outlook}
We now briefly mention two works that build directly on the results of this paper.
The semigroup C*-algebras that we consider here carry canonical time evolutions coming from the norm map on $K$, and a computation of the KMS and ground states of the associated C*-dynamical systems is worked out in \cite{Bru}. There, the finite group $\I_\m/i(K_{\m,\Gamma})$, which appears first in Section~\ref{sec:construction} below, plays an important role. For instance, for each $\beta>2$, the simplex of KMS$_\beta$ states for the canonical time evolution on $C_\lambda^*(R\rtimes R_{\m,\Gamma})$ decomposes over $\I_\m/i(K_{\m,\Gamma})$, whereas uniqueness for $\beta$ in the critical interval $[1,2]$ relies on classical properties of the $L$-functions associated with characters of $\I_\m/i(K_{\m,\Gamma})$, see \cite[Theorem~3.2]{Bru}.

Another natural problem is to determine whether the analyses from \cite{LiLu,Li4,Li5} on $K$-theoretic invariants can be carried out for C*-algebras arising from actions of congruence monoids on rings of algebraic integers. This is investigated in \cite{BruLi}, where we show that the left regular semigroup C*-algebra $C_\lambda^*(R\rtimes R_{\m,\Gamma})$ contains subtle number-theoretic information about $K$ and about a certain class field (i.e., finite abelian extension) of $K$ that is naturally associated with the data $(\m,\Gamma)$, see \cite[Theorem~5.5]{BruLi}. Even in the case of the full $ax+b$-semigroup over $R$, said theorem is novel since no connection with class field theory had been made previously. It is further shown in \cite[Section~3]{BruLi} that $C_\lambda^*(R\rtimes R_{\m,\Gamma})$ is purely infinite in a very strong sense.

\subsection{Organization of this paper}
We begin in Section~\ref{sec:preliminaries} with a brief discussion of notation and preliminaries for semigroup C*-algebras in Section~\ref{subsec:semigpC*} and for moduli of algebraic number fields in Section~\ref{subsec:moduli}. In Section~\ref{sec:construction}, we define $R\rtimes R_{\m,\Gamma}$ and take a first step towards understanding $C_\lambda^*(R\rtimes R_{\m,\Gamma})$; namely, we compute the semilattice of constructible right ideals of $R\rtimes R_{\m,\Gamma}$ and prove that this semilattice satisfies the independence condition from \cite{Li2}, see Proposition~\ref{prop:CI}. This puts us in a setting where we can use general results from Li's theory of semigroup C*-algebras from \cite{Li2,Li3} (see also \cite{Li6} and \cite[Chapter~5]{CELY}). 

We begin our study of the left regular C*-algebra $C_\lambda^*(R\rtimes R_{\m,\Gamma})$ in Section~\ref{subsec:presentation} where we give two presentations for $C_\lambda^*(R\rtimes R_{\m,\Gamma})$ in terms of explicit generators and relations, see Propositions~\ref{prop:full=reduced}~and~\ref{prop:CDLfamily}. In Section~\ref{subsec:gpoidmodel},  we realize $C_\lambda^*(R\rtimes R_{\m,\Gamma})$ as a full corner in a crossed product and hence also as the C*-algebra of a groupoid, see Equation~\eqref{eqn:crossedproduct} and Proposition~\ref{prop:gpoidisom}. Then, in Section~\ref{sec:faithfulreps}, we follow the approach of \cite[Theorem~3.7]{LR} to establish a faithfulness criterion for representations of $C_\lambda^*(R\rtimes R_{\m,\Gamma})$ in terms of spanning projections of the canonical diagonal sub-C*-algebra, see Theorem~\ref{thm:faithful}.

Section~\ref{sec:primideals} contains an explicit description of the primitive ideal space of $C_\lambda^*(R\rtimes R_{\m,\Gamma})$, which generalizes \cite[Theorem~3.6]{EL}, see Theorem~\ref{thm:primideals}. However, in the proof of Theorem~\ref{thm:primideals}, we use a general result by Sims and Williams for groupoid C*-algebras, see \cite[Lemma~4.6]{SW}, rather than working with crossed product C*-algebras as in \cite{EL}. We also give an explicit presentation of the primitive ideals using relations that only involve the range projections of the generating isometries. This presentation is motivated by the description of the primitive ideals of $C_\lambda^*(R\rtimes R^\times)$ given in \cite[Section~3]{Li4} and \cite{Li5}. We then prove in Section~\ref{sec:bq} that the boundary quotient of $C_\lambda^*(R\rtimes R_{\m,\Gamma})$ can be realized as a semigroup crossed product; this generalizes the semigroup crossed product description for the ring C*-algebra of $R$.

In Section~\ref{sec:functoriality}, we show that the number-theoretic input for our construction carries a canonical partial order, and that our construction respects this order, that is, it is functorial in the appropriate sense, see Propositions~\ref{prop:func1}~and~\ref{prop:func2}.

\subsection*{Acknowledgments.} 
I am grateful to my PhD supervisor, Marcelo Laca, for providing lots of helpful comments and feedback on the content and style of this article.
I would also like to thank Xin Li and Mak Trifkovi\'{c} for many helpful discussions and to thank the anonymous referee for several useful suggestions/comments and for mentioning the papers \cite{LarLi} and \cite{BOS}.

\section{Preliminaries}\label{sec:preliminaries}

\subsection{The left regular C*-algebra of a semigroup.}\label{subsec:semigpC*}
Let $P$ be a unital subsemigroup of a countable group $G$, and let $\{\delta_x : x\in P\}$ be the canonical orthonormal basis for $\ell^2(P)$. Each $p\in P$ gives rise to an isometry $\lambda_p$ in $\mathcal{B}(\ell^2(P))$ such that $\lambda_p(\delta_x)=\delta_{px}$ for all $x\in P$.  The \emph{left regular C*-algebra of $P$} is $C_\lambda^*(P):=C^*(\{\lambda_p : p\in P\})$. The canonical ``diagonal'' sub-C*-algebra of $C_\lambda^*(P)$ is $D_\lambda(P):=C_\lambda^*(P)\cap \ell^\infty(P)$, where we view $\ell^\infty(P)$ as sub-C*-algebra of $\mathcal{B}(\ell^2(P))$ in the canonical way. 
Since $P$ embeds into a group, $D_\lambda(P)$ coincides with the smallest unital sub-C*-algebra of $\ell^\infty(P)$ that is invariant under conjugation by the isometries $\lambda_p$ for $p\in P$ and the co-isometries $\lambda_p^*$ for $p\in P$; however, to see this we must introduce some ideas from \cite{Li2}.

For each subset $X\subseteq P$ and $p\in P$, let 
\[
pX:=\{px : x\in X\}\quad \text{ and }\quad p^{-1}(X):=(p^{-1}X)\cap P=\{p^{-1}x\in G : x\in X\}\cap P. 
\]
Consider the smallest collection $\J_P$ of subsets of $P$ such that
\begin{itemize}
	\item $\emptyset$ and $P$ are in $\J_P$;
	\item if $X$ is in $\J_P$ and $p$ is in $P$, then $pX$ and $p^{-1}(X)$ are in $\J_P$;
	\item if $X,Y\in\J_P$, then $X\cap Y\in\J_P$.
\end{itemize}
It is shown in \cite[Section 3]{Li2} that the first two conditions imply the third. Members of $\J_P$ are called \emph{constructible right ideals of $P$}, see \cite[Section~2]{Li2} and \cite[Definition~2.1]{Li3}. We refer the reader to \cite{Li2} or \cite[Section~A.2]{Li1} for a discussion of the motivation for considering constructible ideals and some of the history leading up to their conception.

Since $P$ embeds in a group, the results of \cite[Section~3]{Li2} show that
\[
D_\lambda(P)=\overline{\spn}(\{E_X : X \in \J_P\})
\]
where $E_X\in \mathcal{B}(\ell^2(P))$ is the orthogonal projection onto the subspace $\ell^2(X)\subseteq\ell^2(P)$. At this point, it is not difficult to see that $D_\lambda(P)$ is indeed the smallest unital sub-C*-algebra $D$ of $\ell^\infty(P)$ such that $p\in P$ and $d\in D$ implies $\lambda_pd\lambda_p^*\in D$ and  $\lambda_p^*d\lambda_p\in D$.

Following \cite[Definition~2.26]{Li2}, we say that $\J_P$ is \emph{independent} or $P$ satisfies the \emph{independence condition} if $\bigcup_{i=1}^mX_i=X$ for $X,X_1,...,X_m\in \J_P$ implies $X=X_i$ for some $1\leq i\leq m$. Semigroups satisfying the independence condition are particularly tractable; indeed, if $P$ satisfies the independence condition, then the diagonal C*-algebra $D_\lambda(P)$ enjoys a certain universal property, which we will discuss in Section~\ref{subsec:presentation}. Much of Section~\ref{sec:construction} is devoted to establishing that the class of semigroups under consideration in this paper satisfy the independence condition.

\subsection{Moduli and ray classes.} \label{subsec:moduli}

Let $K$ be a number field with ring of integers $R$, and let $R^\times:=R\setminus \{0\}$ denote the multiplicative semigroup of non-zero elements in $R$. 
Let $\P_K$ denote the set of all non-zero prime ideals of $R$, and let $\I$ denote the group of fractional ideals of $K$.
For $\a\in \I$, there is a unique factorization $\a=\prod_{\p\in\P_K}\p^{v_\p(\a)}$ where $v_\p(\a)\in\mathbb{Z}$, and $v_\p(\a)=0$ for all but finitely many $\p$; for $x\in K^\times:=K\setminus\{0\}$, we let $v_\p(x):=v_\p(xR)$. Let $i:K^\times\to \I$ be the group homomorphism $i(x):=xR$; the \emph{ideal class group of $K$} is given by $\Cl(K):=\I/i(K^\times)$.

If $[K:\mathbb{Q}]$ is the degree of $K$ over $\mathbb{Q}$, then there are exactly $[K:\mathbb{Q}]$ embeddings of $K$ into the complex numbers; these come in two flavours: there are the \emph{real embeddings} $w:K\hookrightarrow\mathbb{R}$ and the \emph{complex embeddings} $w:K\hookrightarrow\mathbb{C}$ such that $w(K)\nsubseteq \mathbb{R}$. We let $V_{K,\mathbb{R}}$ be the (finite) set of real embeddings of $K$.
A \emph{modulus} $\m$ for $K$ is a function $\m:V_{K,\mathbb{R}}\sqcup\P_K\to\mathbb{N}$ such that
\begin{itemize}
\item $\m_\infty:=\m\vert_{V_{K,\mathbb{R}}}:V_{K,\mathbb{R}}\to \mathbb{N}$ takes values in $\{0,1\}$;
\item $\m\vert_{\P_K}:\P_K\to\mathbb{N}$ is finitely supported, that is, $\m(\p)=0$ for all but finitely many $\p$.
\end{itemize}

Let $\m_0$ be the ideal $\m_0:=\prod_\p\p^{\m(\p)}$ of $R$. It is conventional to write $\m$ as a formal product $\m=\m_\infty\m_0$. The set of moduli for $K$ carries a canonical partial order; by definition, $\m\leq \n$ if and only if $\m_\infty(w)\leq \n_\infty(w)$ for all $w\in V_{K,\mathbb{R}}$ and $\m(\p)\leq \n(\p)$ for all $\p\in\P_K$; this is nothing more than the usually partial order on $\mathbb{N}$-valued functions. Traditionally, one says that $\m$ \emph{divides} $\n$ if $\m\leq \n$ and writes $\m\mid \n$ instead of $\m\leq \n$. In particular, a prime $\p$ divides $\m$ if and only if $\m(\p)>0$, and a real embedding $w$ divides $\m$ if and only if $\m_\infty(w)=1$. Thus, we will write $w\mid \m_\infty$ to indicate that $\m_\infty$ takes the value one at the real embedding $w$.
The \emph{multiplicative group of residues modulo $\m$} is 
\[
(R/\m)^*:= \prod_{w\mid\m_\infty} \{\pm 1\}\times(R/\m_0)^*.
\]
If $\m_\infty$ is trivial, that is, if $\m(w)=0$ for all real embeddings $w$, then $(R/\m)^*=(R/\m_0)^*$, and if $\m\vert_{\P_K}$ is trivial, so that $\m_0=R$, then $(R/\m)^*= \prod_{w\mid\m_\infty} \{\pm 1\}$. If $\m$ is trivial, then $(R/\m)^*$ is simply the trivial group.

Note that it does not make sense to talk about additive classes modulo $\m$. By the Chinese Remainder Theorem, $(R/\m_0)^*\cong \prod_{\p\mid\m_0}(R/\p^{\m(\p)})^*.$  Let 
\[
R_\m:=\{a\in R^\times : v_\p(a)=0 \text{ for all $\p$ such that }\p\mid\m_0\}
\]
be the multiplicative semigroup of non-zero algebraic integers that are coprime to the ideal $\m_0$. If $a\in R_\m$, then $a$ is invertible modulo $\m_0$, and we define its \emph{residue modulo $\m$} to be
\[
[a]_\m:=((\textup{sign}(w(a)))_{w\mid \m_\infty}, a+\m_0)\in (R/\m)^*,
\]
where $\textup{sign}(t):=t/|t|$ for any non-zero real number $t$. 

\begin{lemma}\label{prop:onto}
The map $R_\m\to (R/\m)^*$ given by $a\mapsto [a]_\m$ is a surjective semigroup homomorphism. 
\end{lemma}
\begin{proof}
It is easy to see that $[ab]_\m=[a]_\m[b]_\m$ for all $a,b\in R_\m$. Let $(\epsilon,b+\m_0)\in (R/\m)^*$. By \cite[Proposition~2.2(i)]{Nar}, the coset $1+\m_0$ contains (infinitely many) elements of any given signature. Thus, we can find $c\in 1+\m_0$ such that $(\textup{sign}(w(bc)))_{w\mid\m_\infty}=\epsilon$. Since $bc\in R_\m$, and $bc+\m_0=b+\m_0$, we have $[bc]_\m=(\epsilon,b+\m_0)$.
\end{proof}

Let $K_\m:=\{a\in K^\times : v_\p(a)=0 \text{ for all } \p\mid \m_0\}$ be the (multiplicative) subgroup of $K^\times$ consisting of non-zero elements of $K$ whose corresponding principal fractional ideal is coprime to $\m_0$.

\begin{lemma}\label{lem:quotientgp}
The group of (left) quotients $R_\m^{-1}R_\m:=\{a/b : a,b\in R_\m\}$ of $R_\m$ in $K^\times$ coincides with $K_\m$. Therefore, the semigroup homomorphism $R_\m\to (R/\m)^*$ given by $a\mapsto [a]_\m$ has a unique extension to a (surjective) group homomorphism $K_\m\to (R/\m)^*$, which we denote by $x\mapsto [x]_\m$. 
\end{lemma}
\begin{proof}
Clearly, $R_\m^{-1} R_\m\subseteq K_\m$. Let $x\in K_\m$. Then $x R=\a/\b$ with $\a$ and $\b$ integral ideals coprime to $\m_0$, and $\a$ and $\b$ represent the same class $\k$ in $\Cl(K)$. Choose an integral ideal $\c$ in $\k^{-1}$ such that $\c$ is coprime to $\m_0$. Then there are $a,b\in R_\m$ such that $\a\c=aR$ and $\b\c=bR$. Now, $x R=\a/\b=\a\c/\b\c=aR/bR$, so that $x=au/b$ for some $u\in R^*$, which shows the reverse inclusion.

If $x\in K_\m$, then by Lemma~\ref{lem:quotientgp}, we can write $x=a/b$ with $a,b\in R_\m$, and $[x]_\m$ is given by $[x]_\m=[a]_\m[b]_\m^{-1}$. A standard argument shows that this gives a well-defined group homomorphism.
\end{proof}

Moduli play a central role in the ideal-theoretic formulation of class field theory, see \cite[Chapter~V]{MilCFT}.
Let $\I_\m$ denote the group of fractional ideals of $K$ that are coprime to $\m_0$, and let $i:K_\m\to \I_\m$ be the canonical homomorphism given by $a\mapsto aR$. Let $K_{\m,1}:=\{x\in K_\m : [x]_\m=1\}$, so that $K_\m/K_{\m,1}\cong (R/\m)^*$. The group $K_{\m,1}$ is called the \emph{ray modulo $\m$}, and the group $\Cl_\m(K):=\I_\m/ i(K_{\m,1})$ is the \emph{ray class group modulo $\m$}. 
Let $R_{\m,1}:=R\cap K_{\m,1}$, let $R^*$ denote the group of units in $R$, and let $R_{\m,1}^*:=R_{\m,1}\cap R^*$ be the group of invertible elements in $R_{\m,1}$. A relationship between ray class groups and the usual ideal class group is demonstrated by the following standard result.

\begin{proposition}[{\cite[Chapter~V,~Theorem~1.7]{MilCFT}}]\label{prop:rayclassgroup}
For every modulus $\m$, there is a five-term exact sequence 
\[
1\to R_{\m,1}^*\to R^*\to (R/\m)^*\to\Cl_\m(K)\to\Cl(K)\to 1.
\]
Hence, $\Cl_\m(K)$ is a finite group of order 
\[
h_\m:=h\cdot [R^*: R_{\m,1}^*]^{-1}\cdot 2^{r_0}\cdot N(\m_0)\prod_{\p\mid\m_0}(1-N(\p)^{-1})
\]
where $h:=|\Cl(K)|$ is the class number of $K$, $r_0$ denotes the number of real embeddings $w$ of $K$ for which $\m(w)=1$, and $N(\p):=|R/\p|$ is the norm of $\p$.
\end{proposition}

\section{Semigroups defined by actions of congruence monoids\\ on rings of algebraic integers}\label{sec:construction}

Let $K$ be a number field with ring of integers $R$, and fix a modulus $\m$ for $K$. For each subgroup $\Gamma$ of $(R/\m)^*$, let 
\[
 R_{\m,\Gamma}:=\{a\in  R_\m : [a]_\m\in\Gamma \}.
\]
Clearly $R_{\m,\Gamma}$ is a subsemigroup of $R_\m$ containing the semigroup $R_{\m,1}= R_{\m,\{1\}}$. For $\Gamma=(R/\m)^*$, we have $R_{\m,\Gamma}=R_\m$. 

\begin{remark}
Semigroups of the form $ R_{\m,\Gamma}$ are called \emph{congruence monoids}, see \cite[Definition~5]{HK} and \cite{G-HK}.
\end{remark}

\begin{proposition}\label{prop:gpofquotients} 
Let $K_{\m,\Gamma}:=\{x\in K_\m : [x]_\m\in\Gamma\}$. Then $K_{\m,\Gamma}=R_{\m,\Gamma}^{-1} R_{\m,\Gamma}$ where $R_{\m,\Gamma}^{-1} R_{\m,\Gamma}$ is the group of (left) quotients of $R_{\m,\Gamma}$ in $K_\m$.
\end{proposition}
\begin{proof}
Clearly, $R_{\m,\Gamma}^{-1} R_{\m,\Gamma}\subseteq K_{\m,\Gamma}$. Let $x\in K_{\m,\Gamma}$. Using Lemma~\ref{lem:quotientgp}, we can write $x=a/b$ with $a,b\in  R_\m$. Since $[x]_\m=[a]_\m[b]_\m^{-1}\in \Gamma$, there exists $\gamma\in\Gamma$ such that $[a]_\m=[b]_\m \gamma$. By Proposition~\ref{prop:onto}, there exists $c\in R_\m$ such that $[c]_\m=[a]_\m^{-1}$. Now, $[ac]_\m=[a]_\m[c]_\m=[1]_\m$ is in $\Gamma$, and $[bc]_\m=([a]_\m \gamma^{-1})[c]_\m=\gamma^{-1}$ is also in $\Gamma$, so we have that $x=a/b=ac/bc$ is in $R_{\m,\Gamma}^{-1} R_{\m,\Gamma}$. 
\end{proof}

The semigroup $R_{\m,\Gamma}$ acts on (the additive group of) $R$ by multiplication, and we form the semi-direct product $R\rtimes R_{\m,\Gamma}$. Explicitly, $R\rtimes R_{\m,\Gamma}$ consists of pairs $(b,a)$ with $b\in R$ and $a\in R_{\m,\Gamma}$, and the product of two such pairs is $(b,a)(d,c):=(b+ad,ac)$. Our first observation about $R\rtimes R_{\m,\Gamma}$ is the following.

\begin{proposition}\label{prop:ore}
The semigroup $R\rtimes R_{\m,\Gamma}$ is left Ore with enveloping group $(R_\m^{-1}R)\rtimes K_{\m,\Gamma}$ where $R_\m^{-1}R=\{\frac{a}{b}\in K : a\in R, b\in R_\m\}$ denotes the localization of the ring $R$ at $R_\m$. That is, the set of left quotients $(R\rtimes R_{\m,\Gamma})^{-1}(R\rtimes R_{\m,\Gamma})$ taken inside $K\rtimes K^\times$ coincides with the group $(R_\m^{-1}R)\rtimes K_{\m,\Gamma}$.
\end{proposition}
\begin{proof}
For $(b,a),(d,c)\in R\rtimes R_{\m,\Gamma}$, we have 
\begin{equation}\label{eqn:quotients}
(b,a)^{-1}(d,c)=(-ba^{-1},a^{-1})(d,c)=\big(\frac{d-b}{a},\frac{c}{a}\big).
\end{equation} 
Hence, $(R\rtimes R_{\m,\Gamma})^{-1}(R\rtimes R_{\m,\Gamma})$ lies in $(R_\m^{-1}R)\rtimes K_{\m,\Gamma}$. 
A direct calculation shows that $(R\rtimes R_{\m,\Gamma})^{-1}(R\rtimes R_{\m,\Gamma})$ is a group. Since $(R_\m^{-1}R)\rtimes K_{\m,\Gamma}=(R_\m^{-1}R\rtimes\{1\})(\{0\}\rtimes K_{\m,\Gamma})$, we will be done once we show that $(R\rtimes R_{\m,\Gamma})^{-1}(R\rtimes R_{\m,\Gamma})$ contains the subgroups $R_\m^{-1}R\rtimes\{1\}$ and $\{0\}\rtimes K_{\m,\Gamma}$

By considering all products in \eqref{eqn:quotients} with $b=d=0$ and using Proposition~\ref{prop:gpofquotients}, we see that $\{0\}\rtimes K_{\m,\Gamma}$ is contained in $(R\rtimes R_\m)^{-1}(R\rtimes R_{\m,\Gamma})$, and by considering all products in \eqref{eqn:quotients} with $a=c$, we see that $(R_{\m,\Gamma}^{-1}R)\rtimes\{1\}$ is contained in $(R\rtimes R_{\m,\Gamma})^{-1}(R\rtimes R_{\m,\Gamma})$. It remains to show that $R_{\m,\Gamma}^{-1}R$ coincides with $R_\m^{-1}R$. The inclusion $R_{\m,\Gamma}^{-1}R\subseteq R_\m^{-1}R$ is easy to see. Now suppose that $a\in R$ and $b\in R_\m$. By Lemma~\ref{prop:onto}, there is a $c\in R_\m$ such that $[c]_\m=[b]_\m^{-1}$, that is, $w(bc)>0$ for all $w\mid\m_\infty$ and $bc\in 1+\m_0$, so that $bc\in R_{\m,1}$. Now $a/b=ac/bc$ lies in $R_{\m,\Gamma}^{-1}R$, so $R_\m^{-1}R\subseteq R_{\m,\Gamma}^{-1}R$.
\end{proof}

We now turn to the problem of computing the semilattice $\J_{R\rtimes R_{\m,\Gamma}}$ of constructible right ideals in $R\rtimes R_{\m,\Gamma}$. Recall that $\I_\m$ is, by definition, the group of fractional ideals of $K$ that are coprime to $\m_0$. Let $\I_\m^+$ be the submonoid of $\I_\m$ consisting of (non-zero) integral ideals that are coprime to $\m_0$. For $\a\in\I_\m$, we set $\a^\times:=\a\setminus\{0\}$. When $\m=\m_0=R$, we will write $\I$ instead of $\I_R$.

Our goal now is to prove the following result, which generalizes the computation of $\J_{R\rtimes R^\times}$ from \cite[Section~2.4]{Li2}.

\begin{proposition}\label{prop:CI}
The set $\left(\bigsqcup_{\a\in\I_\m^+}R/\a\right)\sqcup\{\emptyset\}$ is a semilattice with respect to intersections. For each $x\in R$ and $\a\in\I_\m^+$, the set $(x+\a)\times(\a\cap  R_{\m,\Gamma})$ is a constructible right ideal of $R\rtimes R_{\m,\Gamma}$, and the map
\[
\left(\bigsqcup_{\a\in\I_\m^+}R/\a\right)\sqcup\{\emptyset\} \to\J_{ R\rtimes  R_{\m,\Gamma}}
\] 
given by $x+\a\mapsto (x+\a)\times(\a\cap  R_{\m,\Gamma})$ and $\emptyset\mapsto \emptyset$ is an isomorphism of semilattices. Moreover, $\J_{R\rtimes  R_{\m,\Gamma}}$ is independent.
\end{proposition}

We need several preliminary results before we can prove Proposition~\ref{prop:CI}. They are contained in the following propositions and lemmas, several of which will also be useful later.

Recall that an element $x\in K^\times$ is \emph{totally positive} if $w(x)>0$ for every real embedding $w:K\hookrightarrow \mathbb{R}$. Note that if $K$ has no real embeddings, then every element of $K^\times$ is totally positive.

\begin{lemma}\label{lem:approx}
Let $\p_1,...,\p_k$ be distinct non-zero primes of $ R$ not dividing $\m_0$ and $n_1,...,n_k$ be in $\mathbb{N}$. There is an element $x$ in $R_{\m,1}$ such that $x$ is totally positive and $v_{\p_j}(x)=n_j$ for $j=1,...,k$.
\end{lemma}
\begin{proof}
For each $1\leq j \leq k$, let $\pi_{\p_j}\in\p_j\setminus \p_j^2$. By the Chinese Remainder Theorem, there exists $y\in R$ such that 
\begin{enumerate}
	\item $y\;\equiv\; \pi_{\p_j}^{n_j}\;\mod\; \p_j^{n_j+1}$;
	\item $y\;\equiv\; 1 \;\mod\; \m_0$.
\end{enumerate}	
The first condition says that $v_{\p_j}(y)=n_j$ for $1\leq j\leq k$. Choose an integer $T$ in $\m_0\p_1^{n_1+1}\cdots \p_k^{n_k+1}$ such that $x:=y+T$ is totally positive. Since $T\in\m_0\p_j^{n_j+1}=\m_0\cap\p_j^{n_j+1}$ for each $1\leq j\leq k$, $x$ still satisfies $(1)$ and $(2)$, so we are done.
\end{proof}

The following two lemmas are refinements of well-known results for the case of trivial $\m$ (in which case $\Gamma$ must also be trivial), see \cite[Lemma~4.15(a)]{CDL} and \cite[Section~2.4]{Li2}.

\begin{lemma}\label{lem:twoGen}
Let $\a\in \I_\m^+$. For each $a\in \a\cap  R_{\m,1}$, there exists $b\in \a\cap  R_{\m,1}$ such that $\a=a R+b R$.
\end{lemma}
\begin{proof}
Write $a R=\a\c_a$ for some ideal $\c_a$ of $ R$. Since $a$ is relatively prime to $\m_0$, we have $\c_a\in \I_\m^+$. By Lemma~\ref{lem:approx}, we can find $b\in \a\cap  R_{\m,1}$ such that $v_\p(b)=v_\p(\a)$ for every prime $\p$ dividing $\c_a$. Now write $b R=\a\c_b$ for some ideal $\c_b$ of $ R$. Since $v_\p(\c_b)=v_\p(b)-v_\p(\a)=0$ for all $\p$ dividing $\c_a$, we see that $\c_a$ and $\c_b$ are relatively prime, that is, $ R=\c_a+\c_b$. Thus, $\a=\a R=\a(\c_a+\c_b)=a R+b R$.
\end{proof}

\begin{lemma}\label{lem:cutdown}
Let $\a\in \I_\m^+$. For each $a\in R_{\m,1}$, there exists $b\in  R_{\m,1}$ such that $\a=\frac{a}{b} R\cap R$.
\end{lemma}
\begin{proof}
Write $a R=\a\c_a$ for some ideal $\c_a$ of $ R$. Since $a\in \c_a$, Lemma~\ref{lem:twoGen} implies that there is a $b\in \c_a\cap  R_{\m,1}$ such that $\c_a=a R+b R$. Since $ab R=(a R+b R)(a R\cap b R)$, we have $\a=a R(\c_a)^{-1}=a R(a R+b R)^{-1}=b^{-1}(a R\cap b R)=\frac{a}{b} R\cap R$.
\end{proof}

For any set $X\subseteq R$, we denote by $X_+$ the subset of all totally positive elements in $X$, and by $\langle  X\rangle  $ the ideal of $R$ generated by $X$.

\begin{lemma}\label{lem:determinedbyray}
Let $\a\in \I_\m^+$. Then for each subgroup $\Gamma\subseteq (R/\m)^*$, $\a$ is generated as an ideal by the set $\a\cap R_{\m,\Gamma}$. Indeed, $\a$ is generated by the set $(\a\cap(1+\m_0))_+=\a\cap  (1+\m_0)_+$.
\end{lemma}
\begin{proof}
Since $\a$ and $\m_0$ are coprime, $\a\cap\m_0=\a\m_0$, and there exists $x\in\a$ and $y\in \m_0$ such that $1=x+y$. Choose an integer $T\in \a\cap\m_0$ such that $x_0:=x+T$ is totally positive. Then $1=x_0+y_0$ with $x_0\in\a$ and $y_0:=y-T\in\m_0$.
	Now, 
\[
	\a\cap(1+\m_0)=
	\begin{cases}
z+\a\cap\m_0=z+\a\m_0 & \text{ if there exists } z\in \a\cap(1+\m_0)\\
\emptyset & \text{ otherwise.}
	\end{cases}
	\]
Hence, $\a\cap(1+\m_0)=x_0+\a\m_0$. Since $x_0\in (x_0+\a\m_0)_+$, it follows that $(\a\m_0)_+$ is contained in $\langle  (x_0+\a\m_0)_+\rangle  $. 

If $\b$ is any non-zero ideal of $R$ and $x$ an element of $\b$, then for sufficiently large $k\in \mathbb{N}^\times$, $x+kN(\b)$ is totally positive. Since $N(\b)\in\b_+$, and $x=(x+kN(\b))-kN(\b)$, we see that any element of a non-zero ideal of $R$ can be written as the difference of two totally positive elements each lying in the ideal. Using this fact, we see that $(\a\m_0)_+\subseteq\langle  (x_0+\a\m_0)_+\rangle  $ implies that $\langle  (x_0+\a\m_0)_+\rangle  $ contains $\a\m_0$. 

Since $\a\supseteq \langle  (x_0+\a\m_0)_+\rangle  $, we will be done if we show that $\m_0$ and $\langle  (x_0+\a\m_0)_+\rangle  $ are coprime.
	Since $x_0\in (x_0+\a\m_0)_+$, it suffices to show that $v_\p(x_0)=0$ for each $\p\mid \m_0$. Let $\p\mid\m_0$. Then we have $0=v_\p(1-x_0+x_0)\geq \min\{v_\p(1-x_0),v_\p(x_0)\}$. Now, $1-x_0=y_0\in\m_0\subseteq\p$, which implies that $v_\p(1-x_0)>0$. Hence, we must have $v_\p(x_0)=0$.
\end{proof}

\begin{proposition}\label{prop:constructibleIdeals}
The set $\I_\m^+\sqcup\{\emptyset\}$ is a semilattice with respect to intersections. For each $\a\in\I_\m^+$, the set $\a\cap  R_{\m,\Gamma}$ is a constructible right ideal of the multiplicative semigroup $R_{\m,\Gamma}$, and the map $\I_\m^+\sqcup\{\emptyset\}\to\J_{R_{\m,\Gamma}}$ given by $\a\mapsto \a\cap  R_{\m,\Gamma}$ and $\emptyset\mapsto\emptyset$ is an isomorphism of semilattices. Moreover, $\J_{ R_{\m,\Gamma}}$ is independent.
\end{proposition}
\begin{proof}
It is clear that $\I_\m^+\sqcup\{\emptyset\}$ is a semilattice with respect to intersections. Now let $\a\in \I_\m^+$. By Lemma~\ref{lem:cutdown}, there exists $a,b\in  R_{\m,1}$ such that $\a=\frac{a}{b} R\cap R$, and so have $\a\cap R_{\m,\Gamma}=\frac{a}{b} R\cap R_{\m,\Gamma}$. If $x\in  R$ such that $\frac{a}{b}x\in R_{\m,\Gamma}$, then $x$ lies in $R_{\m,\Gamma}$; it follows that $\a\cap R_{\m,\Gamma}=\frac{a}{b} R_{\m,\Gamma}\cap R_{\m,\Gamma}$, which clearly lies in $\J_{ R_{\m,\Gamma}}$. This settles the second claim.

To show surjectivity, it suffices to show that $\J:=\{\a\cap R_{\m,\Gamma} : \a\in\I_\m^+\}\subseteq \J_{R_{\m,\Gamma}}\cup\{\emptyset\}$ satisfies the characterizing properties of $\J_{R_{\m,\Gamma}}$ (see Section~\ref{subsec:semigpC*}). Clearly, $\emptyset$ and $R_{\m,\Gamma}$ lie in $\J$. Let $\a\in \I_\m^+$ and $x\in  K_{\m,\Gamma}$. If $xa=y\in R_{\m,\Gamma}$ for some $a\in\a$, then $[a]_\m=[x]_\m^{-1}[y]_\m\in\Gamma$, so $a\in R_{\m,\Gamma}$. Thus,
\[
x(\a\cap R_{\m,\Gamma})\cap  R_{\m,\Gamma}=x\a\cap xR_{\m,\Gamma}\cap R_{\m,\Gamma}=x\a\cap  R_{\m,\Gamma}=(x\a\cap R) \cap R_{\m,\Gamma}
\]
 lies in $\J$, which proves that $\J$ satisfies the desired properties. Hence, $\J_{ R_{\m,\Gamma}}\subseteq\J$ which shows that the map $\a\mapsto \a\cap R_{\m,\Gamma}$ is surjective.  

Suppose now that $\a\cap R_{\m,\Gamma}=\b\cap R_{\m,\Gamma}$ for $\a,\b\in \I_\m^+$. Then Lemma~\ref{lem:determinedbyray} implies that $\a=\b$, so this map is also injective. 

It remains to show independence. Suppose that $\a,\a_1,...,\a_k\in \I_\m^+$ are distinct ideals such that $\a_i\cap  R_{\m,\Gamma}\subseteq \a\cap R_{\m,\Gamma}$ for $i=1,...,k.$ We need to show that $\bigcup_{i=1}^k\a_i\cap R_{\m,\Gamma}\subsetneq \a\cap  R_{\m,\Gamma}$. By Lemma~\ref{lem:determinedbyray}, the inclusion $\a_i\cap  R_{\m,\Gamma}\subseteq \a\cap R_{\m,\Gamma}$ implies that $\a_i\subseteq\a$. Since $\a_i\neq\a$, we even have $\a_i\subsetneq \a$ for $1\leq i\leq k$. Thus, there are positive integers $N\leq M$, distinct non-zero primes $\p_1,...,\p_N,\p_{N+1},...,\p_M$, and non-negative integers $n_1,...n_N,n_{i,1},...,n_{i,M}$, for $1\leq i\leq k$, with $n_j\leq n_{i,j}$ for all $1\leq j\leq N$, $1\leq i\leq M$, such that
\[
\a=\p_1^{n_1}\cdots \p_N^{n_N} \quad \text{ and } \quad \a_i=\p_1^{n_{i,1}}\cdots\p_N^{n_{i,N}}\cdots \p_M^{n_{i,N}}.
\]
By Lemma~\ref{lem:approx}, there exists $x\in  R_{\m,1}$ such that $v_{\p_j}(x)=n_j$ for $j=1,...,N$ and $v_{\p_i}(x)=0$ for $i=N+1,...,M$. It follows that $x\in \a$ and $x\not\in \a_i$ for $i=1,...,k$. Thus, $x\in \a\cap R_{\m,1}\setminus \bigcup_{i=1}^k\a_i$. Since $ R_{\m,1}\subseteq R_{\m,\Gamma}$, it follows that $x\in \a\cap R_{\m,\Gamma}\setminus \bigcup_{i=1}^k\a_i$, so we are done.
\end{proof}

We are now ready to prove Proposition~\ref{prop:CI}.

\begin{proof}[Proof of Proposition~\ref{prop:CI}]
If $x+\a,y+\b$ lie in $\bigsqcup_{\a\in\I_\m^+}R/\a$, then 

\[
(x+\a)\cap (y+\b)=\begin{cases}
z+\a\cap\b & \text{ if } z\in (x+\a)\cap (y+\b) \\
\emptyset & \text{ if } (x+\a)\cap (y+\b)=\emptyset.
\end{cases}
\]
Thus, $\left(\bigsqcup_{\a\in\I_\m^+}R/\a\right)\sqcup\{\emptyset\}$  is a semilattice with respect to intersections

Let $\a\in\I_\m^+$. By Lemma~\ref{lem:cutdown}, we can write $\a=\frac{a}{b}R\cap R$ for some $a,b\in R_{\m,1}$. As in the proof of Proposition~\ref{prop:constructibleIdeals}, we have $\a\cap R_{\m,\Gamma}=\frac{a}{b}R_{\m,\Gamma}\cap R_{\m,\Gamma}$. Thus,
\[
[(0,\frac{a}{b})R\rtimes R_{\m,\Gamma}]\cap R\rtimes R_{\m,\Gamma}=\a\times(\a\cap  R_{\m,\Gamma}),
\]
and for $x\in R$ we have $(x,0)(\a\times(\a\cap  R_{\m,\Gamma}))=(x+\a)\times(\a\cap  R_{\m,\Gamma})$. Hence, $(x+\a)\times(\a\cap  R_{\m,\Gamma})$ is in $\J_{R\rtimes R_{\m,\Gamma}}$ for all $x\in R$ and $\a\in \I_\m^+$. 

To show surjectivity, it suffices to show that $\tilde{\J}:=\{(x+\a)\times(\a\cap  R_{\m,\Gamma}): x\in R,\;\a\in \I_\m^+\}\cup\{\emptyset\}$ satisfies the characterizing properties of $\J_{R\rtimes R_{\m,\Gamma}}$ (see Section~\ref{subsec:semigpC*}).
It is easy to see that $\tilde{\J}$ is closed under taking finite intersections. Let $(x+\a)\times(\a\cap  R_{\m,\Gamma})\in\tilde{\J}$ and $(b,a)\in  R\rtimes R_{\m,\Gamma}$.
Then  
\[
(b,a)[(x+\a)\times(\a\cap  R_{\m,\Gamma})]=(b+ax+a\a)\times (a\a\cap R_{\m,\Gamma})
\]
lies in $\tilde{\J}$. Moreover, for any $c\in R_{\m,\Gamma}$,
\[
(0,c)^{-1}[(x+\a)\times(\a\cap  R_{\m,\Gamma})]\cap (R\rtimes R_{\m,\Gamma})=[c^{-1}(x+\a)\times(c^{-1}\a\cap  R_{\m,\Gamma})]\cap ( R\rtimes R_{\m,\Gamma}).
\]
Now, $c^{-1}(x+\a)\cap R=c^{-1}((x+\a)\cap cR)$ is either empty or of the form $c^{-1}z+c^{-1}\a\cap R$ for some $z\in (x+\a)\cap cR$, and $c^{-1}(\a\cap  R_{\m,\Gamma})\cap R_{\m,\Gamma}=c^{-1}\a\cap R_{\m,\Gamma}.$ It follows that $\J$ satisfies the conditions in Section~\ref{subsec:semigpC*}, which concludes the proof of surjectivity.

Injectivity follows as in the proof of Proposition~\ref{prop:constructibleIdeals}, and independence of $\J_{R\rtimes R_{\m,\Gamma}}$ follow from independence of $\J_{ R_{\m,\Gamma}}$.
\end{proof}

We conclude this section by giving several corollaries. The first simply says that Proposition~\ref{prop:CI} generalizes the computation of $\J_{R\rtimes R^\times}$ from \cite[Section~2.4]{Li2}.

\begin{corollary}[{\cite[Section~2.4]{Li2}}]
We have
\[
\J_{ R\rtimes  R^\times}=\{(x+\a)\times\a^\times: x\in R,\;\a\in \I^+\}\cup\{\emptyset\}
\]
where $\a^\times:=\a\setminus\{0\}$. Moreover, $\J_{ R\rtimes  R^\times}$ is independent.
\end{corollary}
\begin{proof}
 Apply Proposition~\ref{prop:CI} for the case of trivial $\m$ and $\Gamma$.
\end{proof}

As before, let $i:K_\m\to \I_\m$ denote the map $i(x)=xR$. Then the group $\I_\m/i(K_{\m,\Gamma})$ is a quotient of the finite group $\Cl_\m(K)$, hence is finite; indeed, $\I_\m/i(K_{\m,\Gamma})\cong\Cl_\m(K)/\bar{\Gamma}$ where $\bar{\Gamma}=i(K_{\m,\Gamma})/i(K_{\m,1})$. Recall that a semigroup is right LCM if all of its constructible right ideals are principal. 

\begin{corollary}\label{cor:rightLCM}
The semigroup $R\rtimes R_{\m,\Gamma}$ is right LCM if and only if the group $\I_\m/i(K_{\m,\Gamma})$ is trivial.
\end{corollary}
\begin{proof}
By Proposition~\ref{prop:CI}, $R\rtimes R_{\m,\Gamma}$ is right LCM  if and only if every integral ideal $\a\in\I_\m^+$ is principal and generated by some $a\in R_{\m,\Gamma}$. This is equivalent to $\I_\m/i(K_{\m,\Gamma})$ being trivial.
\end{proof}

Let $K=\mathbb{Q}$, so that $R=\mathbb{Z}$. Let $m\in\mathbb{N}^\times$ be a positive natural number, and let $\m=\m_\infty\m_0$ where $\m_\infty$ takes the value one at the only real embedding of $\mathbb{Q}$ and $\m_0(p):=v_p(m)$. Then a calculation shows that $\I_\m/i(K_{\m,1})\cong (\mathbb{Z}/m\mathbb{Z})^*$. Thus, Corollary~\ref{cor:rightLCM} shows that, even in the case $K=\mathbb{Q}$, the semigroup $R\rtimes R_{\m,\Gamma}$ is usually not right LCM.

We also have: 

\begin{corollary}\label{cor:CI}
The map $\J_{R\rtimes  R_{\m,\Gamma}}\to \J_{ R\rtimes  R^\times}$ given by $(x+\a)\times (\a\cap R_{\m,\Gamma})\mapsto (x+\a)\times \a^\times$ and $\emptyset\mapsto \emptyset$ is an injective map of semilattices.
\end{corollary}
\begin{proof}
The map $(x+\a)\times (\a\cap R_{\m,\Gamma})\mapsto (x+\a)\times \a^\times$ is well-defined by Proposition~\ref{prop:constructibleIdeals}, and it is not difficult to see that it is a map of semilattices. Injectivity follows from Proposition~\ref{prop:CI}.
\end{proof}

\section{Presentations for $C_\lambda^*(R\rtimes R_{\m,\Gamma})$.}\label{subsec:presentation}

Let $K$ be a number field with ring of integers $R$. Also let $\m$ be a modulus for $K$, let $S:=\{\p\in\P_K : \p\mid \m_0\}$ be the \emph{support} of $\m_0$, and let $\Gamma\subseteq (R/\m)^*$ be a subgroup. These will remain fixed throughout this section.

We begin with a short discussion of semigroup crossed products. Let $P$ be a subsemigroup of a countable group $G$ as in Section~\ref{subsec:semigpC*}, and suppose that $\alpha$ is an action of $P$ on a unital C*-algebra $D$ by injective *-endomorphisms. The triple $(D,P,\alpha)$ is called a \emph{semigroup dynamical system}. A \emph{covariant representation of $(D,P,\alpha)$} in a unital C*-algebra $B$ is a pair $(\pi,V)$ where $\pi:D\to B$ is a unital *-homomorphism, and $V:P\to \textup{Isom}(B)$ is a semigroup homomorphism satisfying the \emph{covariance condition}
 \begin{equation}\label{eqn:covariance}
V_p\pi(d)V_p^*=\pi(\alpha_p(d))\quad \text{ for all } p\in P \text{ and } d\in D. 
 \end{equation}
 Here, $\textup{Isom}(B)$ denotes the semigroup of isometries in $B$.
 Given a semigroup dynamical system $(D,P,\alpha)$, the \emph{semigroup crossed product $D\rtimes_\alpha P$}, as defined in \cite[Definition~2.2]{LR}, is the universal unital C*-algebra for covariant representations of $(D,P,\alpha)$; that is, $D\rtimes_\alpha P$ is a unital C*-algebra, and there is a covariant representation $(i_D,v)$ of $(D,P,\alpha)$ in $D\rtimes_\alpha P$ such that 
\begin{itemize}
\item $D\rtimes_\alpha P=C^*(\{i_D(d) : d\in D\}\cup\{v_p : p\in P\})$;
\item for any covariant representation $(\pi,V)$ of $(D,P,\alpha)$ in a C*-algebra $B$, there exists a representation $\pi\times V: D\rtimes_\alpha P\to B$ such that $(\pi\times V)\circ i_D=\pi$ and $(\pi\times V)\circ v=V$.
\end{itemize}

Following \cite{Li2}, we now show how to canonically associate a semigroup dynamical system with $P$. By definition, a \emph{semilattice} is a commutative semigroup in which every element is an idempotent; the collection $\J_P$ is a semilattice with semigroup operation given by intersection of subsets.
The \emph{C*-algebra of $\J_P$}, as defined in \cite[Section~2]{LiNor}, is the universal C*-algebra $C_u^*(\J_P)$ generated by projections $\{e_X : X\in \J_P\}$ such that 
\[
e_\emptyset=0 \quad\text{ and }\quad e_Xe_Y=e_{X\cap Y}\quad  \text{ for all }X,Y\in\J_P.
\]

Note that $C_u^*(\J_P)$ is unital with unit $e_P$. Since the collection $\{e_X : X\in \J_P\}$ of generating projections is closed under multiplication, we have $C_u^*(\J_P)=\overline{\spn}(\{e_X : X\in \J_P\})$. The universal property of $C_u^*(\J_P)$ implies existence of a *-homomorphism $C_u^*(\J_P)\to D_\lambda(P)$ determined on the spanning projections by $e_X\mapsto E_X$ where $E_X\in\mathcal{B}(\ell^2(P))$ is, as in Section~\ref{subsec:semigpC*}, the orthogonal projection from $\ell^2(P)$ onto $\ell^2(X)\subseteq\ell^2(P)$. By \cite[Proposition~2.24]{Li2}, this map is an isomorphism if and only if $P$ satisfies the independence condition.

The semigroup $P$ acts on the semilattice $\J_P$ by left multiplication, $p: X\mapsto pX$, which gives rise to an action of $P$ on the (commutative) C*-algebra $C_u^*(\J_P)$ of the semilattice $\J_\P$ by injective *-endomorphisms $\alpha_p$ that are determined on the generating projections by $\alpha_p(e_X)=e_{pX}$. Thus, we get the semigroup dynamical system $(C_u^*(\J_P),P,\alpha)$. 
From the definition of $C_u^*(\J_P)$ we see that the crossed product $C_u^*(\J_P)\rtimes_\alpha P$ is the universal C*-algebra generated by isometries $\{v_p : p\in P\}$ and projections $\{e_X : X\in \J_P\}$ such that
\begin{enumerate}[label=(\Roman*)]
\item\label{relation:(I)} $v_pv_q=v_{pq}$ and $v_pe_Xv_p^*=e_{pX}$ for all $p,q\in P$ and $X\in \J_P$; 
\item\label{relation:(II)} $e_\emptyset=0$, $e_P=1$, and $e_Xe_Y=e_{X\cap Y}$ for all $X,Y\in\J_P$.
\end{enumerate}
This is precisely the presentation for the \emph{(full) semigroup C*-algebra $C^*(P)$ of $P$} as given in \cite[Definition~2.2]{Li2}, so $C^*(P)=C_u^*(\J_P)\rtimes_\alpha P$, see \cite[Lemma~2.14]{Li2}.

Let $\lambda:p\mapsto \lambda_p\in \textup{Isom}(C_\lambda^*(P))$ be the left regular representation of $P$, and let $\eta$ be the canonical *-homomorphism $\eta:C_u^*(\J_P)\to D_\lambda(P)$ such that $\eta(e_X)=E_X$. Then the pair $(\eta,\lambda)$  is a covariant representation of $(C_u^*(\J_P),P,\alpha)$ in $C_\lambda^*(P)$. The associated representation $C^*(P)\to C_\lambda^*(P)$ determined by $v_p\mapsto \lambda_p$ and $e_X\mapsto E_X$ is called the \emph{left regular representation of $C^*(P)$}.

We now turn to the special case of $P_{\m,\Gamma}:=R\rtimes R_{\m,\Gamma}$. First, note that by Proposition~\ref{prop:CI}, the semigroup $P_{\m,\Gamma}$ satisfies the independence condition, so \cite[Proposition~2.24]{Li2} asserts that the canonical *-homomorphism $C_u^*(\J_{P_{\m,\Gamma}})\to D_\lambda(P_{\m,\Gamma})$ is an isomorphism.

\begin{proposition}\label{prop:full=reduced}
The left regular representation $C^*(P_{\m,\Gamma})\to C_\lambda^*(P_{\m,\Gamma})$ is an isomorphism.
\end{proposition}
\begin{proof}
By Proposition~\ref{prop:ore}, $P_{\m,\Gamma}$ is left Ore with solvable, hence amenable, enveloping group $(R_\m^{-1}R)\rtimes K_{\m,\Gamma}$, and $P_{\m,\Gamma}$ satisfies the independence condition by Proposition~\ref{prop:CI}; hence, \cite[Section~3.1]{Li2} combined with \cite[Theorem~6.1]{Li3} implies our claim.
\end{proof}

From now on, we will use Proposition~\ref{prop:full=reduced} to identify $C^*(P_{\m,\Gamma})=C_u^*(\J_{P_{\m,\Gamma}})\rtimes P_{\m,\Gamma}$ with $C_\lambda^*(P_{\m,\Gamma})$. We also have:

\begin{proposition}\label{prop:inclusionofsemilattices}
The canonical inclusion of semilattices $\J_{P_{\m,\Gamma}}\to \J_{R\rtimes R^\times}$ from Corollary~\ref{cor:CI} gives rise to an injective *-homomorphism $C_u^*(\J_{P_{\m,\Gamma}})\to C_u^*(\J_{R\rtimes R^\times})$ such that $e_{(x+\a)\times (\a\cap R_{\m,\Gamma})}\mapsto e_{(x+\a)\times\a^\times}$. Moreover, this map is equivariant for the obvious $P_{\m,\Gamma}$-actions.
\end{proposition}
\begin{proof}
Existence of such a *-homomorphism follows immediately from the universal property of $C_u^*(\J_{P_{\m,\Gamma}})$. Equivariance is obvious, and injectivity follows Proposition~\ref{prop:CI} and \cite[Proposition~5.6.21]{CELY}.
\end{proof}

To avoid cumbersome notation, we will often identify $C_u^*(\J_{P_{\m,\Gamma}})$ with its image in $C_u^*(\J_{R\rtimes R^\times})$ under the canonical inclusion from Proposition~\ref{prop:inclusionofsemilattices}. Thus, we will write $e_{(x+\a)\times\a^\times}$  rather than $e_{(x+\a)\times(\a\cap R_{\m,\Gamma})}$ for a canonical spanning projection of $C_u^*(\J_{P_{\m,\Gamma}})$.

Our next result gives a presentation for $C^*(P_{\m,\Gamma})$ that is, for the particular case of trivial $\m$, entirely analogous to the presentation given in \cite[Definition~2.1]{CDL}, see also \cite[Section~2.4]{Li2}.

\begin{proposition}\label{prop:CDLfamily}
For $x\in R$, let  $u^x:=v_{(x,1)}$, for $a\in R_{\m,\Gamma}$, let $s_a:=v_{(0,a)}$, and for $\a\in\I_\m^+$, let $e_\a:=e_{\a\times(\a\cap R_{\m,\Gamma})}$. Then:
\begin{enumerate}
\item[\textup{(Ta)}] The $u^x$ are unitary and satisfy $u^xu^y=u^{x+y}$, the $s_a$ are isometries and satisfy $s_as_b=s_{ab}$. Moreover, $s_au^x=u^{ax}s_a$ for all $x,y\in R$ and $a,b\in R_{\m,\Gamma}$.
\item[\textup{(Tb)}] The $e_\a$ are projections and satisfy $e_\a e_\b=e_{\a\cap \b}$, $e_R=1$.
\item[\textup{(Tc)}] We have $s_ae_\b s_a^*=e_{a\b}$. 
\item[\textup{(Td)}] For $\a\in\I_\m^+$, $\begin{cases}
u^xe_\a=e_\a u^x & \text{ for } x\in\a, \text{ and} \\
e_\a u^x e_\a=0 & \text{ for } x\not\in\a.
\end{cases}$
\end{enumerate}
Moreover, $C^*(P_{\m,\Gamma})$ is universal in the following sense: if $B$ is a C*-algebra containing elements $U^x$ for $x\in R$, $S_a$ for $a\in R_{\m,\Gamma}$, and $E_\a$ for $\a\in \I_\m^+$ satisfying the obvious ``uppercase'' analogues of \textup{(Ta)}--\textup{(Td)}, then there is a unique *-homomorphism $C^*(P_{\m,\Gamma})\to B$ such that $u^x\mapsto U^x$, $s_a\mapsto S_a$, and $e_\a\mapsto E_\a$.
\end{proposition}
\begin{proof}
A calculation analogous to that given in \cite[Section~2.4]{Li2} shows that the relations \textup{(Ta)}--\textup{(Td)} are satisfied. 

If $\{U^x :x\in R\}$, $\{S_a : a\in R_{\m,\Gamma}\}$, and $\{E_\a : \a\in\I_\m^+\}$ are elements in a C*-algebra $B$ satisfying ``uppercase'' analogues of \textup{(Ta)}--\textup{(Td)}, let $V_{(x,a)}:=U^xS_a$ and $E_{x+\a}:=U^xE_\a U^{-x}$ for $x\in R$, $a\in R_{\m,\Gamma}$, and $\a\in\I_\m^+$. A calculation verifies that these elements satisfying the defining relations \ref{relation:(I)} and \ref{relation:(II)} for $C^*(P_{\m,\Gamma})$, so the universal property of $C^*(P_{\m,\Gamma})$ gives us a *-homomorphism $C^*(P_{\m,\Gamma})\to B$ such that $v_{(x,a)}\mapsto V_{(x,a)}$ and $e_{(x+\a)\times \a^\times}\mapsto E_{x+\a}$.
\end{proof}

\section{Description as a full corner in a crossed product.}\label{subsec:gpoidmodel}

We will now describe $C^*(R\rtimes R_{\m,\Gamma})$ as a full corner in a crossed product, and thus also as a groupoid C*-algebra. Since $R\rtimes R_{\m,\Gamma}$ is left Ore by Proposition~\ref{prop:ore}, this could be derived from \cite[Theorem 2.1.1]{La}. However, for the present setting, the results from \cite[Section 4]{Li3} give us a concrete realization of the ``dilated system'' which will be more convenient for our purposes.

\subsection{The Toeplitz condition}
Let $P$ be a subsemigroup of a group $G$ as in Section~\ref{subsec:semigpC*}. Let $\lambda^G$ denote the left regular representation of $G$ on $\ell^2(G)$, and for each subset $Y\subseteq G$, let $E_Y\in\mathcal{B}(\ell^2(G))$ be the corresponding multiplication operator, that is, $E_Y$ is the orthogonal projection onto $\ell^2(Y)\subseteq \ell^2(G)$.
Let $\J_{P\subseteq G}$ be the smallest collection of subsets of $G$ that contains $\J_P$, is closed under left translation by elements in $G$, and is closed under finite intersections. Let $D_{P\subseteq G}:=\overline{\spn}(\{E_X : X\in \J_{P\subseteq G}\}).$
Then $D_{P\subseteq G}$ is a sub-C*-algebra of $\ell^\infty(G)$, and, as explained in \cite[Section~2.5]{CEL1}, we can identify $D_{P\subseteq G}\rtimes_r G$ with the sub-C*-algebra of $\mathcal{B}(\ell^2(G))$ given by $\overline{\spn}(\{E_Y\lambda_g^G : Y\in \J_{P\subseteq G}, g\in G\})).$
By \cite[Lemma~3.8]{Li3}, the projection $E_P$ is full in $D_{P\subseteq G}\rtimes_r G$. We always have the containment $C_\lambda^*(P)\subseteq E_P(D_{P\subseteq G}\rtimes_r G)E_P$, where we view $C_\lambda^*(P)$ as a sub-C*-algebra of $\mathcal{B}(\ell^2(G))$. The reverse containment need not hold in general.
By \cite[Definition~4.1]{Li3}, the inclusion $P\subseteq G$ satisfies the \emph{left Toeplitz condition} provided that for each $g\in G$, the compression $E_P\lambda_g^GE_P$ of $\lambda_g^G$ by $E_P$ is either zero or of the form $E_P\lambda_g^GE_P=\lambda_{p_1}^*\lambda_{q_1}\cdots\lambda_{p_n}^*\lambda_{q_n}$ for some $p_1,q_1,...,p_n,q_n\in P$. If $P\subseteq G$ satisfies the left Toeplitz condition, then \cite[Lemmas~3.9]{Li3} guarantees that $C_\lambda^*(P)=E_P(D_{P\subseteq G}\rtimes_r G)E_P$.

Now assume that $P\subseteq G$ satisfies the left Toeplitz condition, and let $\Omega_{P\subseteq G}:=\Spec(D_{P\subseteq G})$. By \cite[Lemma~4.2(i)]{Li3}, we have $D_P=E_PD_{P\subseteq G}E_P$, so there is a canonical inclusion $\Omega_P\subseteq\Omega_{P\subseteq G}$, and $E_P(D_{P\subseteq G}\rtimes_r G)E_P\cong 1_{\Omega_P}(C_0(\Omega_{P\subseteq G})\rtimes_r G)1_{\Omega_P}$. 
We also have 
\[
1_{\Omega_P}(C_0(\Omega_{P\subseteq G}\rtimes_r G))1_{\Omega_P}\cong C_r^*(G\ltimes\Omega_P)
\] 
where $G\ltimes\Omega_P:=\{(g,w)\in G\times\Omega_P : gw\in\Omega_P\}$ is the reduction of the transformation groupoid $G\ltimes\Omega_{P\subseteq G}$ by the compact open set $\Omega_P$. Our notation for the reduction groupoid is justified by the fact that the groupoid $G\ltimes\Omega_P$ can be canonically identified with the transformation groupoid for a canonical \emph{partial action} of $G$ on $\Omega_P$, see \cite[Section~2]{Li7}. 

We now return to the case of $R\rtimes R_{\m,\Gamma}\subseteq (R_\m^{-1}R)\rtimes K_{\m,\Gamma}$. To avoid cumbersome notation, we let $P_{\m,\Gamma}:=R\rtimes R_{\m,\Gamma}$ and $G_{\m,\Gamma}:=(R_\m^{-1}R)\rtimes K_{\m,\Gamma}$.  Since $P_{\m,\Gamma}$ is left Ore by Proposition~\ref{prop:ore}, the inclusion $P_{\m,\Gamma}\subseteq G_{\m,\Gamma}$ satisfies the left Toeplitz condition by \cite[Section~8.3]{Li3}. 
From the discussion above, we have isomorphisms
\[
C_\lambda^*(P_{\m,\Gamma})\cong 1_{\Omega_{P_{\m,\Gamma}}}(C_0(\Omega_{P_{\m,\Gamma}\subseteq G_{\m,\Gamma}})\rtimes_r G_{\m,\Gamma})1_{\Omega_{P_{\m,\Gamma}}}\cong C_r^*(G_{\m,\Gamma}\ltimes \Omega_{P_{\m,\Gamma}}).
\]

Our aim now is to describe the diagonal sub-C*-algebra $D_{P_{\m,\Gamma}\subseteq G_{\m,\Gamma}}\cong C_0(\Omega_{P_{\m,\Gamma}\subseteq G_{\m,\Gamma}})$.

\begin{proposition}\label{prop:cosets}
We have $\J_{P_{\m,\Gamma}\subseteq G_{\m,\Gamma}}=\{(x+\a)\times\a^\times : x\in K, \a\in\I_\m\}\cup\{\emptyset\}$.
\end{proposition}
\begin{proof}
Since $P_{\m,\Gamma}\subseteq G_{\m,\Gamma}$ is left Toeplitz, \cite[Lemma~4.2]{Li3} implies that $\J_{P_{\m,\Gamma}\subseteq G_{\m,\Gamma}}=\{gX : g\in G, X\in\J_{P_{\m,\Gamma}}\}$. Hence, $\J_{P_{\m,\Gamma}\subseteq G_{\m,\Gamma}}=\{(y+\a)\times\a^\times : y\in R_\m^{-1}R, \a\in\I_\m\}\cup\{\emptyset\}$, so the inclusion ``$\subseteq$'' holds.

To prove the reverse inclusion, let $\a\in\I_\m$ and $y\in K$. We need to find $x\in R_\m^{-1}R$ such that $x+\a=y+\a$. By strong approximation (\cite[Theorem~6.28]{Nar}), there exists $x\in K$ such that 
\begin{itemize}
\item $v_\p(x-y)\geq v_\p(\a)$ for all $\p\mid \a$;
\item $v_\p(x)\geq 0$ for all $\p\mid\m_0$.
\end{itemize}
That is, $x+\a=y+\a$ and $x$ is integral at every prime that divides $\m_0$. 
Write $xR=\b/\c$ where $\b$ and $\c$ are coprime integral ideals. Then, because $v_\p(x)\geq 0$ for all $\p\mid\m_0$, $\c$ is coprime to $\m_0$ and thus defines a class $[\c]$ in $\I_\m/i(K_\m)$; let $\mathfrak{d}$ be an integral ideal in the inverse class $[\c]^{-1}$, so that $\c\mathfrak{d}=bR$ for some $b\in R_\m$. The class of $\mathfrak{d}$ in $\Cl(K)$ coincides with the inverse of the class of $\c$ in $\Cl(K)$, and $\b$ and $\c$ are in the same ideal class in $\Cl(K)$, so there exists $a\in R$ such that $\b\mathfrak{d}=aR$. Now we have
\[
xR=\b/\c=\b\mathfrak{d}/\c\mathfrak{d}=aR/bR=(a/b)R,
\]
so $x=au/b$ for some $u\in R^*$ which shows that $x\in R_\m^{-1}R$. Since $x+\a=y+\a$, we are done.
\end{proof}

\subsection{An adelic description of the spectrum of the diagonal}
We will now describe $C(\Omega_{P_{\m,\Gamma}})$ and $C_0(\Omega_{P_{\m,\Gamma}\subseteq G_{\m,\Gamma}})$ as functions on certain adelic spaces; this is motivated by \cite[Section~1]{LN2} and \cite[Section~5]{CDL}, also see \cite[Section~2]{Li4}.

Each non-zero prime ideal $\p$ of $R$ defines a normalized \emph{absolute value $|\cdot|_\p$} on $K^\times$; explicitly, $|x|_\p:=N(\p)^{-v_\p(x)}$. We let  $K_\p$ denote the corresponding completion of $K$ and $R_\p=\{x\in K_\p : |x|_\p\leq 1\}$ the ring of integers in $K_\p$. The \emph{ring of finite adeles over $K$} is
\[
\mathbb{A}_f:=\big\{\mathbf{a}=(a_\p)_\p\in \prod_\p K_\p : a_\p\in R_\p\text{ for all but finitely many }\p\big\}.
\]
Equipped with the restricted product topology with respect to the compact open subsets $R_\p\subseteq K_\p$, $\mathbb{A}_f$ is a locally compact ring. Let $\hat{R}$ denote the compact subring $\prod_\p R_\p$ consisting of \emph{integral adeles}. We can modify this definition to work with only the primes not dividing $\m$. Let $S:=\{\p\in\P_K : \p\mid\m_0\}$ be the support of $\m_0$, and put
\[
\mathbb{A}_S:=\big\{\mathbf{a}=(a_\p)_\p\in \prod_{\p\not\in S}K_\p : a_\p\in R_\p \text{ for all but finitely many }\p\big\}.
\]
Also equip $\mathbb{A}_S$ with the restricted product topology. Denote by $\hat{R}_S$ the compact subring $\prod_{\p\not\in S}R_\p$ of $\mathbb{A}_S$, and let $\hat{R}_S^*:=\prod_{\p\not\in S}R_\p^*$ be the group of units in $\hat{R}_S$. The compact group $\hat{R}_S^*$ acts on $\mathbb{A}_S$ by multiplication, and we will let $\bar{\mathbf{a}}$ denote the image of $\mathbf{a}\in\mathbb{A}_S$ under the quotient map $\mathbb{A}_S\to \mathbb{A}_S/\hat{R}_S^*$.
There is a diagonal embedding of additive groups $K\hookrightarrow \mathbb{A}_S$, so $K$ acts on $\mathbb{A}_S$ by translation. Moreover, the image of $K_{\m,\Gamma}$ under this embedding is contained in the multiplicative group $\mathbb{A}_S^*$ of units in $\mathbb{A}_S$, so $K_{\m,\Gamma}$ acts on $\mathbb{A}_S$ by multiplication. This action descends to an action of $K_{\m,\Gamma}$ on the quotient $\mathbb{A}_S/\hat{R}_S^*$ given by $k\bar{\mathbf{a}}=\overline{k\mathbf{\a}}$. Hence, the locally compact space $\mathbb{A}_S\times\mathbb{A}_S/\hat{R}_S^*$ carries a canonical action of $G_{\m,\Gamma}$ given by $(n,k)(\mathbf{b},\bar{\mathbf{a}})=(n+k\mathbf{b},k\bar{\mathbf{a}})$.

\begin{remark}
The space $\hat{R}_S/\hat{R}_S^*$ can be canonically identified with $\prod_{\p\notin S} \p^{\mathbb{N}\cup\{\infty\}}$, which may be thought of as the space of ``super ideals coprime to $\m_0$'', and we can identify $\I_\m^+$ with its canonical image in $\hat{R}_S/\hat{R}_S^*$ via $\a\mapsto \prod_\p\p^{v_\p(\a)}$. Similarly, $\mathbb{A}_S/\hat{R}_S^*$ may be thought of as the space of ``super fractional ideals coprime to $\m_0$''.
\end{remark}

We define an equivalence relation on $\mathbb{A}_S\times\mathbb{A}_S/\hat{R}_S^*$ by $(\mathbf{b},\bar{\mathbf{a}})\sim(\mathbf{d},\bar{\mathbf{c}})$ if $\bar{\mathbf{a}}=\bar{\mathbf{c}}$ and $\mathbf{b}-\mathbf{d}\in \bar{\mathbf{a}}\hat{R}_S$. The action of $G_{\m,\Gamma}$ descends to a well-defined action on the locally compact quotient space 
\[
\Omega_K^\m:=(\mathbb{A}_S\times\mathbb{A}_S/\hat{R}_S^*)/\sim.
\]
This equivalence relation restricts to an equivalence relation on the compact subset $\hat{R}_S\times\hat{R}_S/\hat{R}_S^*\subseteq \mathbb{A}_S\times\mathbb{A}_S/\hat{R}_S^*$, and the quotient space 
\[
\Omega_R^\m:=(\hat{R}_S\times\hat{R}_S/\hat{R}_S^*)/\sim
\] 
is a compact subset of $\Omega_K^\m$. 

\begin{proposition}\label{prop:adelicdescription}
There are $G_{\m,\Gamma}$-equivariant isomorphisms $D_{P_{\m,\Gamma}}\cong C(\Omega_R^\m)$ and $D_{P_{\m,\Gamma}\subseteq G_{\m,\Gamma}}\cong C_0(\Omega_K^\m)$ such that the following diagram commutes
\[\xymatrix{D_{P_{\m,\Gamma}}\ar[d]^\cong \ar@{^{(}->}[r] & D_{P_{\m,\Gamma}\subseteq G_{\m,\Gamma}}\ar[d]^{\cong}\\
C(\Omega_R^\m) \ar@{^{(}->}[r] & C_0(\Omega_K^\m)}
\]
where the horizontal arrows are the canonical inclusions, and the vertical arrows are determined by 
\[
e_{(x+\a)\times\a^\times}\mapsto 1_{\{[\mathbf{b},\bar{\mathbf{a}}] : v_\p(\bar{\mathbf{a}})\geq v_\p(\a)\text{ and } v_\p(\mathbf{b}-x)\geq v_\p(\a) \text{ for all }\p\notin S\}}.
\]
\end{proposition}
\begin{proof}
From Lemma~\ref{prop:cosets}, we have $\J_{P_{\m,\Gamma}\subseteq G_{\m,\Gamma}}=\{(x+\a)\times\a^\times : x\in K, \a\in\I_\m\}\cup\{\emptyset\}$. When $\m$ is trivial, the result follows from the analysis in \cite[Section~2]{Li4}, and the general case goes through almost verbatim.
\end{proof}

An immediate consequence, we have isomorphisms
\begin{equation}\label{eqn:crossedproduct}
C_\lambda^*(P_{\m,\Gamma})\cong 1_{\Omega_R^\m}(C_0(\Omega_K^\m)\rtimes_r G_{\m,\Gamma})1_{\Omega_R^\m}\cong C_\lambda^*(G_{\m,\Gamma}\ltimes\Omega_R^\m)
\end{equation}
where $G_{\m,\Gamma}\ltimes\Omega_R^\m=\{(g,w)\in G_{\m,\Gamma}\times \Omega_R^\m : gw\in \Omega_R^\m\}$ is the reduction groupoid of the transformation groupoid $G_{\m,\Gamma}\ltimes\Omega_K^\m$ with respect to the compact open set $\Omega_R^\m$. 

\begin{proposition}\label{prop:gpoidisom}
There is an isomorphism
\[
\vartheta:C^*(P_{\m,\Gamma})\cong C^*(G_{\m,\Gamma}\ltimes\Omega_R^\m)
\]
that is determined on generators by $\vartheta(v_{(b,a)})=1_{\{(b,a)\}\times\Omega_R^\m}$ for $(b,a)\in P_{\m,\Gamma}$.
\end{proposition}
\begin{proof}
Since $G_{\m,\Gamma}$ is amenable, there is a canonical isomorphism $C^*(G_{\m,\Gamma}\ltimes \Omega_{P_{\m,\Gamma}})\cong C_r^*(G_{\m,\Gamma}\ltimes \Omega_{P_{\m,\Gamma}})$. Hence, the result follows from Proposition~\ref{prop:full=reduced} combined with \eqref{eqn:crossedproduct}.  
\end{proof}

\section{Faithful representations of $C^*(R\rtimes R_{\m,\Gamma})$.} \label{sec:faithfulreps}

\subsection{A criterion for faithfulness}
As before, we will use the notation $P_{\m,\Gamma}:=R\rtimes R_{\m,\Gamma}$ and $G_{\m,\Gamma}:=(R_\m^{-1}R)\rtimes K_{\m,\Gamma}$. Also let $S:=\{\p: \p\mid \m_0\}$ be the support of $\m_0$ and put $\P_K^\m:=\P_K\setminus S$.

Following the approach of \cite[Theorem~3.7]{LR}, we next establish a faithfulness criterion for representations of $C^*(P_{\m,\Gamma})$ in terms of spanning projections of the diagonal.

\begin{theorem}\label{thm:faithful}
For each class $\k\in\I_\m/i(K_{\m,\Gamma})$, choose an integral ideal $\a_\k\in\k$. Suppose $\psi$ is a representation of $C^*(P_{\m,\Gamma})$ in a C*-algebra $B$. Then $\psi$ is injective if and only if for each $\k\in\I_\m/i(K_{\m,\Gamma})$, we have
\[
\psi\left(\prod_{i=1}^m(e_{\a_\k\times \a_\k^\times}-e_{(y_i+\a_i)\times \a_i^\times})\right)\neq 0
\]
for all $y_1,...,y_m\in R$ and $\a_1,...,\a_m\in\I_\m^+$ such that $y_i+\a_i\subsetneq \a_\k$ for $1\leq i\leq m$.
\end{theorem}

We need a preliminary result.

\begin{proposition}\label{prop:faithfulondiagonal}
A representation $\psi$ of $C^*(P_{\m,\Gamma})$ is faithful if and only if it is faithful on $C_u^*(\J_{P_{\m,\Gamma}})$.
\end{proposition}
\begin{proof}
Since the isomorphism $C^*(P_{\m,\Gamma})\cong C^*(G_{\m,\Gamma}\ltimes\Omega_R^\m)$ from Proposition~\ref{prop:gpoidisom} carries $C_u^*(\J_{P_{\m,\Gamma}})$ isomorphically onto $C(\Omega_R^\m)$, it suffices to prove that a representation $\psi$ of the C*-algebra $C^*(G_{\m,\Gamma}\ltimes\Omega_R^\m)$ is faithful if and only if it is faithful on $C(\Omega_R^\m)$.

Since $G_{\m,\Gamma}\ltimes\Omega_R^\m$ is amenable, by \cite[4.4~Theorem]{Ex}, it suffices to show that $G_{\m,\Gamma}\ltimes\Omega_R^\m$ is essentially principal; in the terminology from \cite{Ex}, this means that we need to show that the interior of the isotropy bundle of $G_{\m,\Gamma}\ltimes\Omega_R^\m$ coincides with the unit space of $G\ltimes\Omega_R^\m$. For this, it suffices to show that the set of points in $\Omega_R^\m$ with trivial isotropy is dense in $\Omega_R^\m$; this is a special case of the subsequent result.
\end{proof}

For each $w\in \Omega_R^\m$, let $G_{\m,\Gamma}.w:=\{gw: (g,w)\in \Omega_R^\m\}$ be the orbit of $w$; its closure $\overline{G_{\m,\Gamma}.w}$ is called the \emph{quasi-orbit} of $w$. The following proposition is more than we need; its full strength will be used in Section~\ref{sec:primideals} below.

\begin{proposition}[cf. {\cite[Lemmas 3.1, 3.4 and Corollary 3.5]{EL}}]\label{prop:quasiorbits}
For $\bar{\mathbf{a}}\in \hat{R}_S/\hat{R}_S^*$, let $Z(\bar{\mathbf{a}}):=\{\p\in \P_K^\m : \bar{\mathbf{a}}_\p=0\}$, and for each set $A\subseteq\P_K^\m$, let $C_A:=\{[\mathbf{b},\bar{\mathbf{a}}]\in\Omega_R^\m : A\subseteq Z(\bar{\mathbf{a}})\}$. Then
\begin{enumerate}
\item the quasi-orbit of a point $[\mathbf{b},\bar{\mathbf{a}}]\in  \Omega_R^\m$ is equal to $C_{Z(\bar{\mathbf{a}})}$;
\item for any closed $G_{\m,\Gamma}$-invariant subset $C\subseteq \Omega_R^\m$, the set of points in $C$ with trivial isotropy is dense in $C$.
\end{enumerate}
In particular, the set of points in $\Omega_R^\m$ with trivial isotropy is dense in $\Omega_R^\m$.
\end{proposition}
\begin{proof}
The proof of the first part is similar to the proof of \cite[Lemma~3.1]{EL}, but differs in a few places, so we include it here.
 
Clearly, we have $[\mathbf{b},\bar{\mathbf{a}}]\in C_{Z(\bar{\mathbf{a}})}$. Since $C_{Z(\bar{\mathbf{a}})}$ is closed and $G_{\m,\Gamma}$-invariant, it follows that the quasi-orbit of $[\mathbf{b},\bar{\mathbf{a}}]$ is contained in $C_{Z(\bar{\mathbf{a}})}$. Thus, we only need to show that $C_{Z(\bar{\mathbf{a}})}$ is contained in the quasi-orbit of $[\mathbf{b},\bar{\mathbf{a}}]$. 
	Let $[\mathbf{d},\bar{\mathbf{c}}]\in C_{Z(\bar{\mathbf{a}})}$. Any open set containing $[\mathbf{d},\bar{\mathbf{c}}]$ contains the image under the quotient map $\pi:\hat{R}_S\times \hat{R}_S/\hat{R}_S^*\to\Omega_R^\m$ of an (open) set $W_1\times W_2$ where $W_1\subseteq\hat{R}_S$ is an open set of the form 
\[
W_1=\{\mathbf{e}\in \hat{R}_S : v_\p(\mathbf{e}-\mathbf{d})\geq v_\p(\a)\text{ for all }\p\in\mathcal{P}_K^\m\}
\]
for some integral ideal $\a\in\I_\m^+$, and $W_2\subseteq\hat{R}_S/\hat{R}_S^*$ is an open set of the form
\[
	W_2=\{\bar{\mathbf{e}}\in \hat{R}_S/\hat{R}_S^* : v_\p(\bar{\mathbf{e}})=v_\p(\bar{\mathbf{c}})\text{ for }\p\in F\setminus Z(\bar{\mathbf{c}}) \text{ and } v_\p(\bar{\mathbf{e}})\geq n_\p \text{ for } \p\in F\cap Z(\bar{\mathbf{c}})\}
\]
for some finite set $F\subseteq\P_K^\m$ and non-negative integers $n_\p$ for $\p\in F\cap Z(\bar{\mathbf{c}})$. By Lemma~\ref{lem:approx}, we can find $b\in R_{\m,1}$ such that $v_\p(b)=v_\p(\bar{\mathbf{a}})$ for $\p\in F\setminus Z(\bar{\mathbf{c}})$. Now use Lemma~\ref{lem:approx} again to choose $a\in R_{\m,1}$ such that
\begin{itemize}
		\item $v_\p(a)=v_\p(\bar{\mathbf{c}})$ for $\p\in F\setminus Z(\bar{\mathbf{c}})$;
		\item $v_\p(a)=n_\p+v_\p(b)$ for $\p\in F\cap Z(\bar{\mathbf{c}})$;
		\item $v_\p(a)=v_\p(b)$ for $\p\in F^c$ with $v_\p(b)>0$.
\end{itemize}
	
Let $k:=a/b$. Then $k\in K_{\m,1}$, $k\bar{\mathbf{a}}\in \hat{R}_S/\hat{R}_S^*$, and $k\bar{\mathbf{a}}\in W_2$. By strong approximation (\cite[Theorem~6.28]{Nar}), $K$ is dense in $\mathbb{A}_S$, so there exists $y\in K$ such that $y+k\mathbf{b}\in W_1$. As in the proof of Lemma~\ref{prop:cosets}, we can find $x\in R_\m^{-1}R$ such that $x-y\in\a$. Then $x+k\mathbf{b}\in W_1$, so we have that $(x,k)[\mathbf{b},\bar{\mathbf{a}}]\subseteq \pi(W_1\times W_2)$. Hence, $C_{Z(\bar{\mathbf{a}})}$ is contained in the quasi-orbit of $[\mathbf{b},\bar{\mathbf{a}}]$.
	
An argument analogous to that given in the proof of \cite[Lemma~3.4]{EL} now shows that for any $A\subseteq \P_K^\m$, there exists $[\mathbf{d},\bar{\mathbf{c}}]\in C_A$ such that the isotropy group of $[\mathbf{d},\bar{\mathbf{c}}]$ is trivial and $Z(\bar{\mathbf{c}})=A$. This implies part (2), so we are done.
\end{proof}

We are now ready for the proof of Theorem~\ref{thm:faithful}.

\begin{proof}[Proof of Theorem~\ref{thm:faithful}]
By Proposition~\ref{prop:faithfulondiagonal}, it suffices to prove that the restriction of $\psi$ to $C_u^*(\J_{P_{\m,\Gamma}})$ is injective. For this, by \cite[Proposition~5.6.21]{CELY}, it is enough to show that 
\begin{equation}\label{eqn:defectproj}
\psi\left(e_{(y+\a)\times\a^\times}-\bigvee_{i=1}^me_{(y_i+\a_i)\times \a^\times}\right)=\psi\left(\prod_{i=1}^m(e_{(y+\a)\times\a^\times}-e_{(y_i+\a_i)\times \a^\times})\right)\neq 0
\end{equation}
for $y,y_1,...,y_m\in R$, $\a,\a_1,...,\a_m\in\I_\m^+$ such that $y_i+\a_i\subsetneq y+\a$ for $1\leq i\leq m$. Here, $\bigvee_{i=1}^me_{(y_i+\a_i)\times \a^\times}$ is the smallest projection in $C_u^*(\J_{P_{\m,\Gamma}})$ that dominates each $e_{(y_i+\a_i)\times \a^\times}$, see \cite[Lemma~5.6.20.]{CELY}.

We will exploit the covariance condition, see \eqref{eqn:covariance}. For each $(b,a)\in P_{\m,\Gamma}$, let $W_{(b,a)}:=\psi(v_{(b,a)})$, and observe that
\[\begin{aligned}
W_{(y,1)}^*\psi(\prod_{i=1}^m(e_{(y+\a)\times\a^\times}- e_{(y_i+\a_i)\times \a^\times}))&W_{(y,1)}=\psi(\prod_{i=1}^m(v_{(y,1)}^*(e_{(y+\a)\times\a^\times}- e_{(y_i+\a_i)\times \a^\times})v_{(y,1)})\\
&=\psi(\prod_{i=1}^m(v_{(y,1)}^*e_{(y+\a)\times\a^\times}v_{(y,1)}-v_{(y,1)}^*e_{(y_i+\a_i)\times \a^\times}v_{(y,1)}))\\
&=\psi(\prod_{i=1}^m(e_{\a\times\a^\times}- e_{(y_i-y+\a_i)\times \a^\times})).
\end{aligned}\]
Since $W_{(y,1)}$ is a unitary, it follows that $\psi(\prod_{i=1}^m(e_{(y+\a)\times\a^\times}- e_{(y_i+\a_i)\times \a^\times}))$ is non-zero if and only if  $\psi(\prod_{i=1}^m(e_{\a\times\a^\times}- e_{(y_i-y+\a_i)\times \a^\times}))$ is non-zero; hence, it is enough to show that \eqref{eqn:defectproj} holds when $y=0$.

Let $y_1,...,y_m\in R$ and $\a,\a_1,...,\a_m\in\I_\m^+$ be such that $y_i+\a_i\subsetneq \a$ for $1\leq i\leq m$. If $\k\in\I_\m/i(K_{\m,\Gamma})$ is the class containing $\a$, then there exists $a,b\in R_{\m,\Gamma}$ such that $a\a=b\a_\k$. We have
\[\begin{aligned}
W_{(0,a)}\psi(e_{\a\times\a^\times})W_{(0,a)}^*&=\psi(v_{(0,a)}e_{\a\times\a^\times}v_{(0,a)}^*)\\
&=\psi(e_{a\a\times(a\a)^\times})\\
&=\psi(e_{b\a_\k\times(b\a_\k)^\times})\\
&=\psi(v_{(0,b)}e_{\a_\k\times\a_\k^\times}v_{(0,b)}^*)\\
&=W_{(0,b)}\psi(e_{\a_\k\times\a_\k^\times})W_{(0,b)}^*.
\end{aligned}\] 
Now, $y_i+\a_i\subseteq \a$ implies that $ay_i+a\a_i\subseteq a\a=b\a_\k$. Hence, there exists $\tilde{y}_i\in\a_\k$ such that $ay_i=b\tilde{y}_i$. From this, we see that $a\a_i\subseteq b(\a_\k-\tilde{y}_i)$ which implies that $\tilde{\a}_i:=\frac{a}{b}\a_i$ is an integral ideal. 
Since $a,b\in R_{\m,\Gamma}$, we see also that $\tilde{\a}_i$ is coprime to $\m_0$, so that $\tilde{\a}_i$ lies in $\I^+_\m$. Since $a(y_i+\a_i)=b(\tilde{y}_i+\tilde{\a}_i)$, we have
\begin{align*}
W_{(0,a)}\psi(e_{(y_i+\a_i)\times\a_i^\times})W_{(0,a)}^*&=\psi(e_{a(y_i+\a_i)\times (a\a_i)^\times})\\
&=\psi(e_{b(\tilde{y}_i+\tilde{\a}_i)\times (b\tilde{\a}_i)^\times})\\
&=W_{(0,b)}\psi(e_{(\tilde{y}_i+\tilde{\a}_i)\times \tilde{\a}_i^\times}W_{(0,b)}^*.
\end{align*}
Conjugating by an isometry defines an injective *-homomorphism, so
\[\begin{aligned}
\psi(\prod_{i=1}^n(e_{\a\times\a^\times}-e_{(y_i+\a_i)\times\a_i^\times}))= 0 & \iff W_{(0,a)}\psi(\prod_{i=1}^ne_{\a\times\a^\times}-e_{(y_i+\a_i)\times\a_i^\times})W_{(0,a)}^*=0\\
&\iff \psi(\prod_{i=1}^nv_{(0,a)}e_{\a\times\a^\times}v_{(0,a)}^*-v_{(0,a)}e_{(y_i+\a_i)\times\a_i^\times}v_{(0,a)}^*)=0\\
&\iff \psi(\prod_{i=1}^nv_{(0,b)}e_{\a_\k\times\a_\k^\times}v_{(0,b)}^*-v_{(0,b)}e_{(\tilde{y}_i+\tilde{\a}_i)\times\tilde{\a}_i^\times}v_{(0,b)}^*)=0\\
&\iff W_{(0,b)}\psi(\prod_{i=1}^n(e_{\a_\k\times\a_\k^\times}-e_{(\tilde{y}_i+\tilde{\a}_i)\times\tilde{\a}_i^\times}))W_{(0,b)}^*=0\\
&\iff \psi(\prod_{i=1}^n(e_{\a_\k\times\a_\k^\times}-e_{(\tilde{y}_i+\tilde{\a}_i)\times\tilde{\a}_i^\times}))=0.
\end{aligned}\]
Since $\psi(\prod_{i=1}^n(e_{\a_\k\times\a_\k^\times}-e_{(\tilde{y}_i+\tilde{\a}_i)\times\tilde{\a}_i^\times}))$ is non-zero by assumption, $\psi(\prod_{i=1}^n(e_{\a\times\a^\times}-e_{(y_i+\a_i)\times\a_i^\times}))$ must also be non-zero. Hence, $\psi$ is injective on $C_u^*(\J_{P_{\m,\Gamma}})$ as desired.
\end{proof} 

As an immediate consequence, we obtain the following reformulation.

\begin{corollary}
Suppose that $B$ is a C*-algebra containing elements $U^x$ for $x\in R$, $S_a$ for $a\in R_{\m,\Gamma}$, and $E_\a$ for $\a\in \I_\m^+$ satisfying the ``uppercase'' analogues of \textup{(Ta)}--\textup{(Td)} from Proposition~\ref{prop:CDLfamily}, and let $\psi: C^*(P_{\m,\Gamma})\to B$ be the unique *-homomorphism such that $\psi(u^x)=U^x$, $\psi(s_a)=S_a$, and $\psi(e_\a)=E_\a$.
Then $\psi$ is an isomorphism onto the sub-C*-algebra of $B$ generated by $\{U^x x\in R\},\{S_a : a\in R_{\m,\Gamma}\}$, and $\{E_\a :\a\in\I_\m^+\}$ if and only if for each $\k\in \I_\m/i(K_{\m,\Gamma})$, we have 
\[
\prod_{i=1}^m(E_{\a_\k}-U^{y_i}E_{\a_i}U^{-y_i})\neq 0
\]
for all $y_1,...,y_m\in R$ and $\a_1,...,\a_m\in\I_\m^+$ such that $y_i+\a_i\subsetneq \a_\k$ for $1\leq i\leq m$.
\end{corollary}

Thus, Theorem~\ref{thm:faithful} may be viewed as a uniqueness result, analogous to a Cuntz-Krieger uniqueness theorem.

\subsection{Representations coming from ideal classes}

Using the inclusion from Corollary~\ref{cor:CI}, we will view $\J_{P_{\m,\Gamma}}$ as a subsemilattice of $\J_{R\rtimes R^\times}$. The canonical action of $P_{\m,\Gamma}$ on $\J_{P_{\m,\Gamma}}^\times$ given by $(b,a)[(x+\a)\times\a^\times]=(b+ax+a\a)\times(a\a)^\times$ gives rise to an isometric representation $V$ of $P_{\m,\Gamma}$ on the Hilbert space $\mathcal{H}:=\ell^2(\J_{P_{\m,\Gamma}}^\times)$; namely, $V:P_{\m,\Gamma}\to\textup{Isom}(\mathcal{H})$ is determined on the canonical orthonormal basis by $V_{(b,a)}\delta_{(x+\a)\times\a^\times}=\delta_{(b+ax+a\a)\times(a\a)^\times}$.

\begin{proposition}
For each class $\k\in \I_\m/i(K_{\m,\Gamma})$, the subspace $\mathcal{H}_\k:=\overline{\spn}(\{\delta_{(z+\b)\times\b^\times} : \b\in\k\})\subseteq \mathcal{H}$ is invariant under $V_{(b,a)}$ for all $(b,a)\in P_{\m,\Gamma}$. 
Let $V_{(b,a)}^\k$ be the restriction of $V_{(b,a)}$ to $\mathcal{H}_\k$. For $x\in R$ and $\a\in\I_\m^+$, let $P_{x+\a}^\k$ be the orthogonal projection from $\mathcal{H}_\k$ onto the subspace $\overline{\spn}(\{\delta_{(z+\b)\times\b^\times} : z+\b\subseteq x+\a\})$. 
Then there is a representation $\psi_\k:C^*(P_{\m,\Gamma})\to\mathcal{B}(\mathcal{H}_\k)$ such that $\psi_\k(v_{(b,a)})=V_{(b,a)}^\k$ and $\psi_\k(e_{(x+\a)\times\a^\times})=P^\k_{x+\a}$. Moreover, $\psi_\k$ is faithful.
\end{proposition}
\begin{proof}
It is easy to see that $\mathcal{H}_\k$ is invariant. A calculation shows that the collections $\{V_{(b,a)}^\k: (b,a)\in P_{\m,\Gamma}\}$ and $\{0\}\cup \{P_{x+\a}^\k : x\in R, \a\in\I_\m^+\}$ satisfy the defining relations \ref{relation:(I)} and \ref{relation:(II)} for $C^*(P_{\m,\Gamma})$, so existence of $\psi_\k$ follows from the defining universal property of $C^*(P_{\m,\Gamma})$. For each $\tilde{\k}\in\I_\m/i(K_{\m,\Gamma})$, let $\a_{\tilde{\k}} \in\tilde{\k}$ be an integral ideal. By Theorem~\ref{thm:faithful}, injectivity of $\psi_\k$ will follow if we show that for each $\tilde{\k}$, we have $\prod_{i=1}^m(P_{\a_{\tilde{\k}}}^\k-P_{y_i+\a_i}^\k)\neq 0$ for any $y_1,...,y_m\in R$ and $\a_1,...,\a_m\in\I_\m^+$ such that $y_i+\a_i\subsetneq \a_{\tilde{\k}}$. For this, it suffices to find $\b\in\k$ such that $\b\subseteq \a_{\tilde{\k}}$ and $y_i+\a_i\not\subseteq \b$. By \cite[Theorem~7.2]{MilCFT}, the class $\k\tilde{\k}^{-1}$ contains infinitely many prime ideals, so we can choose a prime $\p\in \k\tilde{\k}^{-1}$ such that $y_1,y_2,...,y_m\not\in\p$. Then $\b:=\p\a_{\tilde{\k}}\in \k$ clearly satisfies $\b\subsetneq\a_{\tilde{\k}}$, and we also have $y_i+\a_i\not\subseteq\b$ because $\b\subseteq \p$, and $y_i+\a_i\subseteq\b$ would imply $y_i\in\b$.
\end{proof}

\begin{remark}
In the case of trivial $\m$, it is shown in \cite[Section~4]{CDL} that the direct sum $\oplus_{\k\in\Cl(K)}\psi_\k$ is faithful.
\end{remark}

\section{The primitive ideal space}\label{sec:primideals}

Given a C*-algebra $B$, let $\Prim(B)$ denote the primitive ideal space of $B$. If $X\subseteq B$ is any subset, we let $\langle X\rangle_B$ denote the (closed, two-sided) ideal of $B$ generated by $X$; by convention, $\langle\emptyset\rangle:=\{0\}$.

Continuing with the notation from the previous section, we let $\P_K^\m:=\P_K\setminus S$ denote the collection of (non-zero) prime ideals of $R$ that do not divide $\m_0$, let $P_{\m,\Gamma}=R\rtimes R_{\m,\Gamma}$, and let $G_{\m,\Gamma}=(R_\m^{-1}R)\rtimes K_{\m,\Gamma}$. 

Equip $2^{\P_K^\m}$ with the power-cofinite topology. Recall that a base for the power-cofinite topology is given by the sets $U_F:=\{T\in 2^{\P_K^\m} : T\cap F=\emptyset\}$ for $F\subseteq 2^{\P_K^\m}$ finite. We may view both $2^{\P_K^\m}$ and $\Prim(C^*(P_{\m,\Gamma}))$ as partially ordered sets with respect to the orders given by inclusion of subsets and inclusion of ideals, respectively. The following theorem is a strengthening and generalization of \cite[Theorem~3.6]{EL}. Our explicit description is motivated by the explicit description of the primitive ideals of $C^*(R\rtimes R^\times)$ given in \cite{Li4,Li5}.

\begin{theorem}\label{thm:primideals}
For each $\p\in \P_K^\m$, let $f_\p$ denote the order of $[\p]\in\I_\m/i(K_{\m,\Gamma})$, so that $\p^{f_\p}=t_\p R$ for some $t_\p\in R_{\m,\Gamma}$.
For each subset $A\subseteq\P_K^\m$, let 
	\[
	I_A:=\left\langle\left\{1-\sum_{x\in R/t_\p R}v_{(x,t_\p)}v_{(x,t_\p)}^* : \p\in A\right\}\right\rangle_{C^*(P_{\m,\Gamma})}.
	\]
Then $I_A$ is a primitive ideal, and the map $2^{\P_K^\m}\to \Prim(C^*(P_{\m,\Gamma}))$ given by $A\mapsto I_A$ is an order-preserving homeomorphism.
\end{theorem}

Before we can prove Theorem~\ref{thm:primideals}, we need a preliminary result.

Each open $G_{\m,\Gamma}$-invariant subset $U\subseteq \Omega_R^\m$ gives rise to the ideal $C^*(G_{\m,\Gamma}\ltimes U)\subseteq C^*(G_{\m,\Gamma}\ltimes\Omega_R^\m)$. In particular, for each point $w\in \Omega_R^\m$, the set $\Omega_R^\m\setminus\overline{G_{\m,\Gamma}.w}$ is open and $G_{\m,\Gamma}$-invariant where, as before, $G_{\m,\Gamma}.w:=\{gw: (g,w)\in \Omega_R^\m\}$ is the orbit of $w$, and $\overline{G_{\m,\Gamma}.w}$ is the closure of $G_{\m,\Gamma}.w$, which is called the quasi-orbit of $w$. The \emph{quasi-orbit space} is given by $\mathcal{Q}(G_{\m,\Gamma}\ltimes\Omega_R^\m):=\Omega_R^\m/\sim$ where $w\sim w'$ if $\overline{G_{\m,\Gamma}.w}=\overline{G_{\m,\Gamma}.w'}$; this space was described in Proposition~\ref{prop:quasiorbits} above.

\begin{lemma}\label{lem:primideals}
For each $x\in\Omega_R^\m$, the ideal $C^*(G_{\m,\Gamma}\ltimes (\Omega_R^\m\setminus \overline{G_{\m,\Gamma}.w}))$ is primitive, and the map $\Omega_R^\m\to \Prim(C^*(G_{\m,\Gamma}\ltimes\Omega_R^\m))$ given by $w\mapsto C^*(G_{\m,\Gamma}\ltimes (\Omega_R^\m\setminus \overline{G_{\m,\Gamma}.w}))$ descends to a homeomorphism $\mathcal{Q}(G_{\m,\Gamma}\ltimes\Omega_R^\m)\simeq \Prim(C^*(G_{\m,\Gamma}\ltimes\Omega_R^\m))$.

Moreover, if $\vartheta:C^*(P_{\m,\Gamma})\cong C^*(G_{\m,\Gamma}\ltimes\Omega_R^\m)$ is the isomorphism from Proposition~\ref{prop:gpoidisom}, then $\vartheta(I_A)=C^*(G_{\m,\Gamma}\ltimes (\Omega_R^\m\setminus C_A))$ for every $A\subseteq\P_K^\m$.
\end{lemma}
\begin{proof}
Each ideal $C^*(G_{\m,\Gamma}\ltimes (\Omega_R^\m\setminus \overline{G_{\m,\Gamma}.w}))$ is primitive by \cite[Lemma~4.5]{SW}.
	
The groupoid $G_{\m,\Gamma}\ltimes\Omega_R^\m$ is second countable, \'{e}tale, and amenable. By Proposition~\ref{prop:quasiorbits}(2), we may apply \cite[Lemma~4.6]{SW} to conclude that the map $\Omega_R^\m\to \Prim(C^*(G_{\m,\Gamma}\ltimes\Omega_R^\m))$ given by $w\mapsto C^*(G_{\m,\Gamma}\ltimes (\Omega_R^\m\setminus \overline{G_{\m,\Gamma}.w}))$ descends to a homeomorphism $\mathcal{Q}(G_{\m,\Gamma}\ltimes\Omega_R^\m)\simeq \Prim(C^*(G_{\m,\Gamma}\ltimes\Omega_R^\m))$.

We now turn to the second claim. Let $A\subseteq\P_K^\m$. For each $\p\in A$, we have that
\[
\vartheta(1-\sum_{x\in R/t_\p R}v_{(x,t_\p)}v_{(x,t_\p)}^*)=1_{\{[\mathbf{b},\bar{\mathbf{a}}] : v_\p(\bar{\mathbf{a}})<f_\p\}}
\]
lies in $C_0(\Omega_R^\m\setminus C_A)$. Hence, $\vartheta(I_A)\subseteq C^*(G_{\m,\Gamma}\ltimes(\Omega_R^\m\setminus C_A))$.
	
We know that $\vartheta(I_A)=\bigcap_J J$ where $J$ runs over all primitive ideals of $C^*(G_{\m,\Gamma}\ltimes\Omega_R^\m)$ that contain $\vartheta(I_A)$, so to show that $C^*(G_{\m,\Gamma}\ltimes (\Omega_R^\m\setminus C_A))$ is contained in $\vartheta(I_A)$, it suffices to show that any primitive ideal that contains $\vartheta(I_A)$ must also contain $C^*(G_{\m,\Gamma}\ltimes(\Omega_R^\m\setminus C_A))$.
Suppose $J\in\Prim(C^*(G_{\m,\Gamma}\ltimes\Omega_R^\m))$ with $\vartheta(I_A)\subseteq J$. By part (1), $J=C^*(G_{\m,\Gamma}\ltimes (\Omega_R^\m\setminus C_B))$ for some $B\subseteq\P_K^\m$. Now, we have $1_{\{[\mathbf{b},\bar{\mathbf{a}}] : v_\p(\bar{\mathbf{a}})<f_\p\}}\in C^*(G_{\m,\Gamma}\ltimes (\Omega_R^\m\setminus C_B))$ for all $\p\in A$ which implies that $1_{\{[\mathbf{b},\bar{\mathbf{a}}] : v_\p(\bar{\mathbf{a}})<f_\p\}}$ vanishes on $C_B$ for all $\p\in A$; hence, $A\subseteq B$. Thus, $C^*(G_{\m,\Gamma}\ltimes (\Omega_R^\m\setminus C_A))\subseteq J$.
\end{proof}

We are now ready to prove Theorem~\ref{thm:primideals}.

\begin{proof}[Proof of Theorem~\ref{thm:primideals}]
By Proposition~\ref{prop:quasiorbits}(1) and Lemma~\ref{lem:primideals}(1), the map $A\mapsto C^*(G_{\m,\Gamma}\ltimes (\Omega_R^\m\setminus C_A))$ is a order-preserving bijection from $2^{\P_K^\m}$ onto $\Prim(C^*(G_{\m,\Gamma}\ltimes \Omega_R^\m)$.  The proof that this map is a homeomorphism is analogous to the proof of \cite[Proposition~2.4]{LR2}. Thus, Theorem~\ref{thm:primideals} follows from Lemma~\ref{lem:primideals}(2).
\end{proof}

\begin{corollary}\label{cor:primideal}
The ideal $I_{\P_K^\m}$ is the unique maximal ideal of $C^*(P_{\m,\Gamma})$, and the map $\p\mapsto I_{\{\p\}}$ defines a bijection from $\P_K^\m$ onto the set of minimal primitive ideals of $C^*(P_{\m,\Gamma})$. 
\end{corollary}
\begin{proof}
This follows from Theorem~\ref{thm:primideals} since the bijection $A\mapsto I_A$ is inclusion-preserving.
\end{proof}

\section{The boundary quotient}\label{sec:bq}
By Corollary~\ref{cor:primideal}, the ideal $I_{\P_K^\m}$ is the unique maximal ideal of $C^*(P_{\m,\Gamma})$. The C*-algebra $C^*(P_{\m,\Gamma})/I_{\P_K^\m}$ is the \emph{boundary quotient} of $C^*(P_{\m,\Gamma})$, as defined in \cite[Section~7]{Li3} (see also \cite[Chapter~5.7]{Li7}). We now give a description of $C^*(P_{\m,\Gamma})/I_{\P_K^\m}$ as a semigroup crossed product. This generalizes the well-known semigroup crossed product description of the ring C*-algebra of $R$.

Each $(b,a)\in P_{\m,\Gamma}$ gives rise to an injective continuous map $\hat{R}_S\to\hat{R}_S$ given by $(b,a)\mathbf{x}:=b+a\mathbf{x}$; let $\tau_{(b,a)}$ be the corresponding *-endomorphism of $C(\hat{R}_S)$. Then $(C(\hat{R}_S),P_{\m,\Gamma},\tau)$ is a semigroup dynamical system, so we may form the crossed product C*-algebra $C(\hat{R}_S)\rtimes_\tau P_{\m,\Gamma}$. For $(b,a)\in P_{\m,\Gamma}$, let $w_{(b,a)}$ be the corresponding isometry in $C(\hat{R}_S)\rtimes_\tau P_{\m,\Gamma}$.

\begin{proposition}\label{prop:boundaryquotient}
There is a surjective *-homomorphism $\pi:C^*(P_{\m,\Gamma})\to C(\hat{R}_S)\rtimes_\tau P_{\m,\Gamma}$ such that 
\begin{equation*}\label{eqn:boundaryquotientmap}
\pi(v_{(b,a)})=w_{(b,a)}\quad \text{ and }\quad \pi(e_{(x+\a)\times\a^\times})=1_{x+\hat{\a}}
\end{equation*}
for all $(b,a)\in P_{\m,\Gamma}$ and $(x+\a)\times\a^\times\in\J_{P_{\m,\Gamma}}^\times$, where $\hat{\a}$ denotes the closed ideal of $\hat{R}_S$ generated by $\a$. Moreover, $\ker\pi=I_{\P_K^\m}$, so we get an isomorphism $C^*(P_{\m,\Gamma})/I_{\P_K^\m}\cong C(\hat{R}_S)\rtimes_\tau P_{\m,\Gamma}$.
\end{proposition}
\begin{proof}
Consider the collection of projections $\{1_{x+\hat{\a}} : x\in R, \a\in\I_\m^+\}$ and the collection of isometries $\{w_{(b,a)} : (b,a)\in P_{\m,\Gamma}\}$. A calculation verifies that these collections satisfy the defining relations \ref{relation:(I)} and \ref{relation:(II)} for $C^*(P_{\m,\Gamma})$, so the universal property of $C^*(P_{\m,\Gamma})$ gives us a *-homomorphism $\pi:C^*(P_{\m,\Gamma})\to C(\hat{R}_S)\rtimes_\tau P_{\m,\Gamma}$ such that \begin{equation*}
\pi(v_{(b,a)})=w_{(b,a)}\quad \text{ and }\quad \pi(e_{(x+\a)\times\a^\times})=1_{x+\hat{\a}}
\end{equation*}
for all $(b,a)\in P_{\m,\Gamma}$ and $(x+\a)\times\a^\times\in\J_{P_{\m,\Gamma}}^\times$.
Since $\spn\{1_{x+\hat{\a}} : x\in R, \a\in\I_\m^+\}$ is dense in $C(\hat{R}_S)$, we see that 
\[
\{1_{x+\hat{\a}} : x\in R, \a\in\I_\m^+\}\cup \{w_{(b,a)} : (b,a)\in P_{\m,\Gamma}\}
\] 
generates $C(\hat{R}_S)\rtimes_\tau P_{\m,\Gamma}$ as a C*-algebra, so $\pi$ is surjective.

It remains to show that $\ker\pi=I_{\P_K^\m}$. Since $I_{\P_K^\m}$ is a maximal ideal, it suffices to show that $I_{\P_K^\m}\subseteq \ker\pi$. For every $\a\in\I_\m^+$, the canonical embedding $R\hookrightarrow\hat{R}_S$ induces an isomorphism $R/\a\cong\hat{R}_S/\hat{\a}$, so $\hat{R}_S=\bigsqcup_{x\in R/\a}(x+\hat{\a})$. Hence,
\[
\pi\big(1-\sum_{x\in R/t_\p R}v_{(x,t_\p)}v_{(x,t_\p)}^*\big)=1-\sum_{x\in R/t_\p R}1_{x+t_\p\hat{R}_S}=0.
\]
Since the projections $1-\sum_{x\in R/t_\p R}v_{(x,t_\p)}v_{(x,t_\p)}^*$ for $\p\in\P_K^\m$ generate $I_{\P_K^\m}$, we are done.
\end{proof}

\section{Functoriality}\label{sec:functoriality}

As before, let $K$ be a number field with ring of integers $R$. Recall that the number-theoretic data for our construction consists of a pair $(\m,\Gamma)$ where $\m$ a modulus for $K$ and $\Gamma$ a subgroup of $(R/\m)^*$. The set of such pairs carries a canonical partial order, which we now describe.

Let $\m$ and $\n$ be moduli for $K$, and let $\Gamma$ and $\Lambda$ be subgroups of $(R/\m)^*$ and $(R/\n)^*$, respectively. Denote by $\textup{pr}_\m:R_\m\to(R/\m)^*$ and $\textup{pr}_\n:R_\n\to(R/\n)^*$ the canonical projection maps. Recall that $\m\mid\n$ if $\m_0\mid\n_0$ and $\m_\infty\leq \n_\infty$. 
If $\m\mid \n$, then we have a canonical inclusion of semigroups $R_\n\subseteq R_\m$, and a canonical surjective group homomorphism $\pi_{\n,\m}:(R/\n)^*\to (R/\m)^*$ such that the following diagram commutes:
\begin{equation}\label{diagram:moduli}
\xymatrix{
R_\n\ar@{^{(}->}[r]^{\textup{incl}}\ar@{->>}[d]^{\textup{pr}_\n} &  R_\m \ar@{->>}[d]^{\textup{pr}_\m}\\
(R/\n)^* \ar@{->>}[r]^{\pi_{\n,\m}}& (R/\m)^*.
}\end{equation}

We define $(\m,\Gamma)\leq (\n,\Lambda)$ if and only if $\m\mid \n$ and $\pi_{\n,\m}(\Lambda)\subseteq \Gamma$. We will show next that our construction respects this ordering, that is, it is functorial in the appropriate sense. First, we need a lemma.

\begin{lemma}\label{lem:detectsw}
Let $\m$ be a modulus for $K$, and suppose that $w$ is a real embedding of $K$. Then $w\mid\m_\infty$ if and only if $w(x)>0$ for all $x\in R_{\m,1}$.
\end{lemma}
\begin{proof}
First, suppose that $w(x)>0$ for all $x\in R_{\m,1}$, and assume that $w\nmid\m_\infty$. By definition, 
\[
R_{\m,1}=\{x\in 1+\m_0 : v(x)>0 \text{ for all } v\mid\m_\infty\},
\]
and \cite[Proposition~2.2(i)]{Nar} asserts that the coset $1+\m_0$ contains (infinitely many) elements of every signature. Hence, there exists $a\in R_{\m,1}$ with $v(a)>0$ for every $v\mid\m_\infty$ and $w(a)<0$. This contradicts that $w(x)>0$ for all $x\in R_{\m,1}$, so we must have $w\mid\m_\infty$. The other direction is obvious.
\end{proof}

\begin{proposition}\label{prop:func1}
Let $\m$ and $\n$ be moduli for $K$, and let $\Gamma$ and $\Lambda$ be subgroups of $(R/\m)^*$ and $(R/\n)^*$, respectively. Then
\begin{enumerate}
\item $R_{\n,\Lambda}\subseteq R_{\m,\Gamma}$ if and only if $(\m,\Gamma)\leq(\n,\Lambda)$.
\item If the equivalent conditions from (1) are satisfied, so that there is a canonical inclusion of semigroups $\iota:R\rtimes R_{\n,\Lambda}\hookrightarrow R\rtimes R_{\m,\Gamma}$, then there is an injective *-homomorphism $C^*(R\rtimes R_{\n,\Lambda})\to C^*(R\rtimes R_{\m,\Gamma})$ such that $v_{(b,a)}\mapsto v_{(\iota(b),\iota(a))}$.
\end{enumerate}
\end{proposition}
\begin{proof}
(1): First, note that $R_{\n,\Lambda}\subseteq R_{\m,\Gamma}$ implies that $R_{\n,1}\subseteq R_{\m,1}$. We will now show that $\m_\infty\leq \n_\infty$. Suppose $w$ is a real embedding of $K$ such that $\m_\infty(w)=1$. Since $R_{\n,1}\subseteq R_{\m,1}$, we must have $w(x)>0$ for all $x\in R_{\n,1}$, so $w\mid \n_\infty$ by Lemma~\ref{lem:detectsw}.

Next we show that $\m_0\mid\n_0$. The inclusion $R_{\n,1}\subseteq R_{\m,1}$ implies that $(1+\n_0)_+\subseteq (1+\m_0)_+$, which in turn implies that $(\n_0)_+\subseteq (\m_0)_+$. Since ideals are generated by the totally positive elements that they contain (see the proof of Lemma~\ref{lem:determinedbyray}), we have $\n_0\subseteq \m_0$. 

Using commutativity of \eqref{diagram:moduli} and that $R_{\n,\Lambda}\subseteq R_{\m,\Gamma}$, we have 
\[
\pi_{\n,\m}(\Lambda)=\pi_{\n,\m}(\textup{pr}_\n(R_{\n,\Lambda}))=\textup{pr}_\m(R_{\n,\Lambda})\subseteq \textup{pr}_\m(R_{\m,\Gamma})=\Gamma,
\]
as desired.

For the converse, suppose $(\m,\Gamma)\leq(\n,\Lambda)$, so that $\m\mid\n$ and $\pi_{\n,\m}(\Lambda)\subseteq \Gamma$. Then $\textup{pr}_\m^{-1}(\pi_{\n,\m}(\Lambda))\subseteq \textup{pr}_\m^{-1}(\Gamma)=R_{\m,\Gamma}$, and commutativity of \eqref{diagram:moduli} implies $\textup{pr}_\m(R_{\n,\Lambda})=\pi_{\n,\m}(\Lambda)$,
so we have $R_{\n,\Lambda}\subseteq \textup{pr}_\m^{-1}(\pi_{\n,\m}(\Lambda))\subseteq R_{\m,\Gamma}$.

(2): Assume $R_{\n,\Lambda}\subseteq R_{\m,\Gamma}$. Then $\m\mid\n$ by part (1) which implies that $\I_\n^+\subseteq \I_\m^+$ . The collections $\{e_{(x+\a)\times \a^\times} : x\in R, \a\in\I_\n^+\}\cup\{0\}$ and $\{v_{(\iota(b),\iota(a))} : (b,a)\in R\rtimes R_{\n,\Lambda}\}$ of projections and isometries, respectively, in $C^*(R\rtimes R_{\m,\Gamma})$ satisfy the defining relations \ref{relation:(I)} and \ref{relation:(II)} for $C^*(R\rtimes R_{\n,\Lambda})$, so the universal property of $C^*(R\rtimes R_{\n,\Lambda})$ gives us a *-homomorphism $\psi:C^*(R\rtimes R_{\n,\Lambda})\to C^*(R\rtimes R_{\m,\Gamma})$ such that $\psi(v_{(b,a)})=v_{(\iota(b),\iota(a))}$ for all $(b,a)\in R\rtimes R_{\n,\Lambda}$.

The projections $\{e_{(x+\a)\times\a^\times} : x\in R, \a\in\I_\n^+\}$ are linearly independent in $C_u^*(\J_{R\rtimes R_{\m,\Gamma}})$ by Proposition~\ref{prop:CI}, so the hypotheses of Theorem~\ref{thm:faithful} are satisfied; hence, $\psi$ is injective.
\end{proof}

In particular, if we take $\m$ to be trivial, so that $\Gamma$ must also be trivial, then we obtain the following result.

\begin{corollary}
For each modulus $\n$ and each subgroup $\Lambda\subseteq (R/\n)^*$, there is an injective *-homomorphism $C^*(R\rtimes R_{\n,\Lambda})\to C^*(R\rtimes R^\times)$ such that $v_{(b,a)}\mapsto v_{(\iota(b),\iota(a))}$ where $\iota:R\rtimes R_{\n,\Lambda}\hookrightarrow R\rtimes R^\times$ is the canonical inclusion.
\end{corollary}

We can also ask what happens as the number field varies. Let $K$ and $K'$ be number fields with rings of integers $R$ and $R'$, respectively.

\begin{lemma}\label{lem:inclusionrays}
Suppose that $\m$ is a modulus for $K$ and that there is an inclusion of number fields $i:K\hookrightarrow K'$. Define a modulus $\tilde{\m}$ for $K'$ by $\tilde{\m}_\infty(w'):=\m_\infty(w'\circ i)$ for each real embedding $w':K'\hookrightarrow\mathbb{R}$ and $\tilde{\m}_0:=i(\m_0)R'$ where $i(\m_0)R'$ is the ideal of $R'$ generated by $i(\m_0)$. 
For each modulus $\m'$ of $K'$, we have $i(R_{\m,1})\subseteq R_{\m',1}'$ if and only if $\m'\mid \tilde{\m}$.
\end{lemma}
\begin{proof}
Suppose that $i(R_{\m,1})\subseteq R_{\m',1}'$. Then for each $w'\mid \m_\infty'$, we see that $w'\circ i(x)>0$ for every $x\in R_{\m,1}$, so $(w'\circ i)\mid\m_\infty$ by Lemma~\ref{lem:detectsw}. That is, $w'\mid \m_\infty'$ implies $w'\mid \tilde{\m}_\infty$, so we have $\m_\infty'\mid\tilde{\m}_\infty$.
The inclusion $i(R_{\m,1})\subseteq R_{\m',1}'$ also implies that $(1+i(\m_0))_+\subseteq (1+\m_0')_+$ where $(1+i(\m_0))_+$ and $(1+\m_0')_+$ denote the sets of totally positive elements in $1+i(\m_0)$ and $1+\m_0'$, respectively. It follows that $\tilde{\m}_0=i(\m_0)R'$ is contained in $\m_0'$, that is, $\m_0'\mid \tilde{\m}_0$. Thus, we have shown $i(R_{\m,1})\subseteq R_{\m',1}'$ implies $\m'\mid\tilde{\m}$.

For the converse, suppose that $\m'\mid \tilde{\m}$. Let $a\in R_{\m,1}$, so that $a\in 1+\m_0$ and $w(a)>0$ for every $w\mid\m_\infty$. We have $1+\tilde{\m}_0\subseteq 1+\m_0'$, and if $w'\mid \m_\infty'$, then $w'\mid\tilde{\m}$, so that $(w'\circ i)\mid \m_\infty$. Hence, $i(a)\in1+\m_0'$, and if $w'\mid\m_\infty'$, then $(w'\circ i)(a)>0$. That is, $i(a)\in R_{\m',1}'$. Hence, $i(R_{\m,1})\subseteq R_{\m',1}'$, as desired.
\end{proof}

In the setup from Lemma~\ref{lem:inclusionrays}, suppose that $\m'\mid \tilde{\m}$. The inclusion $i\vert_R:R\hookrightarrow R'$ induces homomorphisms $(R/\m_0)^*\to (R'/\tilde{\m}_0)^*$ and $\prod_{w\mid\m_\infty}\{\pm1\}\to \prod_{(w'\circ i)\mid \m_\infty}\{\pm1\}$. Combining these, gives us a homomorphism $\varphi:(R/\m)^*\to (R'/\tilde{\m})^*$. These maps give rise to the following commutative diagram
\begin{equation}\label{diagram:moduli2}
\xymatrix{
R_\m \ar@{->>}[d]^{\textup{pr}_\m}\ar@{^{(}->}[r]^{i\vert_{R_\m}} & R_{\tilde{\m}}'\ar@{->>}[d]^{\textup{pr}_{\tilde{\m}}} \ar@{^{(}->}[r]^{\textup{incl}}& R_{\m'}'\ar@{->>}[d]^{\textup{pr}_{\m'}}\\
(R/\m)^*\ar[r]^{\varphi} & (R'/\tilde{\m})^* \ar@{->>}[r]^{\pi_{\tilde{\m},\m'}} & (R'/\m')^*.
}\end{equation}

\begin{proposition}\label{prop:func2}
Let $\m$ and $\m'$ be moduli for $K$ and $K'$, respectively, and let $\Gamma$ and $\Gamma'$ be subgroups of $(R/\m)^*$ and $(R'/\m')^*$, respectively. Suppose that there is an inclusion of number fields $i:K\hookrightarrow K'$. Then, using the notation from the preceding discussion, we have the following: 
\begin{enumerate}
\item $i(R_{\m,\Gamma})\subseteq R_{\m',\Gamma'}'$ if and only if $\m'\mid\tilde{\m}$ and $\pi_{\tilde{\m},\m'}\circ\varphi(\Gamma)\subseteq \Gamma'$.
\item If the equivalent conditions in (1) are satisfied, so that there is an inclusion $\iota:K\rtimes K^\times\hookrightarrow K'\rtimes (K')^\times$ that restricts to an inclusion $R\rtimes R_{\m,\Gamma}\hookrightarrow R'\rtimes R_{\m',\Gamma'}'$,
then there is an injective *-homomorphism $C^*(R\rtimes R_{\m,\Gamma})\hookrightarrow C^*(R'\rtimes R_{\m',\Gamma'}')$ such that $v_{(b,a)}\mapsto v_{(\iota(b),\iota(a))}$ for all $(b,a)\in R\rtimes R_{\m,\Gamma}$.
\end{enumerate}
\end{proposition}
\begin{proof}
(1): Suppose that $i(R_{\m,\Gamma})\subseteq R_{\m',\Gamma'}'$. Then $i(R_{\m,1})\subseteq R_{\m',1}'$, so Lemma~\ref{lem:inclusionrays} implies that $\m'\mid\tilde{\m}$. Let $\gamma\in\Gamma$, and write $\gamma=[a]_\m$ for some $a\in R_{\m,\Gamma}$. Using commutativity of \eqref{diagram:moduli2}, we have 
\[
\pi_{\tilde{\m},\m'}\circ\varphi(\gamma)=\pi_{\tilde{\m},\m'}([i(a)]_{\tilde{\m}})=[i(a)]_{\m'}.
\]
Since $i(a)$ lies in $R_{\m',\Gamma'}'$ by assumption, we have $[i(a)]_{\m'}\in\Gamma'$. Hence, $\pi_{\tilde{\m},\m'}\circ\varphi(\Gamma)\subseteq \Gamma'$.

For the converse, suppose that $\m'\mid\tilde{\m}$ and $\pi_{\tilde{\m},\m'}\circ\varphi(\Gamma)\subseteq \Gamma'$. Let $a\in R_{\m,\Gamma}$. We need to show that $i(a)$ lies in $R_{\m',\Gamma'}'$, that is, we need to show that $[i(a)]_{\m'}$ lies in $\Gamma'$. By commutativity of \eqref{diagram:moduli2}, we have $[i(a)]_{\m'}=\pi_{\tilde{\m},\m'}\circ\varphi([a]_\m)$. Since $[a]_\m\in \Gamma$ and $\pi_{\tilde{\m},\m'}\circ\varphi(\Gamma)\subseteq \Gamma'$, we have $\pi_{\tilde{\m},\m'}\circ\varphi([a]_\m)\in\Gamma'$, as desired.

(2): By \cite[Proposition~3.2~and~Theorem~4.13]{CDL}, there is an injective *-homomorphism $\psi: C^*(R\rtimes R^\times)\to C^*(R'\rtimes (R')^\times)$ such that $\psi(v_{(b,a)})=v_{(\iota(b),\iota(a))}$. Let $\theta$ and $\theta'$ by the canonical injective *-homomorphisms $\theta:C^*(R\rtimes R_{\m,\Gamma})\to C^*(R\rtimes R^\times)$ and $\theta':C^*(R'\rtimes R'_{\Gamma'})\to C^*(R'\rtimes (R')^\times)$ from Proposition~\ref{prop:func1}. There is a (unique) *-homomorphism $\rho$ such that the following diagram commutes:
\[\xymatrix{
C^*(R\rtimes R^\times)\ar@{^{(}->}[r]^{\psi} & \textup{im}(\theta')\ar[d]_{\cong}^{(\theta'\vert_{\textup{im}(\theta')})^{-1}}.\\
C^*(R\rtimes R_{\m,\Gamma})\ar@{^{(}->}[u]^{\theta}\ar@{-->}[r]^{\rho} &  C^*(R'\rtimes R_{\m',\Gamma'}').
}\]
Moreover, it is not difficult to see that $\rho$ is injective and $\rho(v_{(b,a)})=v_{(\iota(b),\iota(a))}$.
\end{proof}

\end{document}